\documentclass[a4paper,12pt]{amsart}
\usepackage{amsmath, amssymb, amsthm, xcolor,enumerate}
\usepackage{verbatim}
\usepackage{hyperref}
\usepackage{xypic}
\usepackage{graphicx}

\usepackage{color}
\definecolor{cutcolour}{RGB}{0,100,0}

\input{xy} 
\xyoption{all} 
\CompileMatrices

\newcounter{theorem}
\newtheorem{theorem}[subsection]{Theorem}
\newtheorem{lemma}[subsection]{Lemma}
\newtheorem{proposition}[subsection]{Proposition}

\newtheorem{introtheorem}[theorem]{Theorem}

\newtheorem{introcorollary}[theorem]{Corollary}

\theoremstyle{definition}
\newtheorem{definition}[subsection]{Definition}

\newtheorem{introdefinition}[theorem]{Definition}

\newtheorem*{remark*}{Remark}
\newtheorem{remark}[subsection]{Remark}

\numberwithin{equation}{section}

\newcommand{\ssum}{\textstyle\sum\limits}

\newcommand{\dimnuc}{\dim_{\mathrm{nuc}}}
\newcommand{\id}{\mathrm{id}}

\allowdisplaybreaks[2] 

\title{Nuclear dimension of simple $\mathrm{C}^*$-algebras}

\author[J.\ Castillejos]{Jorge Castillejos}
\address{\hskip-\parindent Jorge Castillejos, Department of Mathematics, KU Leuven, Celestijnenlaan 200b, 3001 Leuven, Belgium.}
\curraddr{Institute of Mathematics, Polish Academy of Sciences, ul. {\'S}niadeckich 8, 00-656 Warszawa, Poland}
\email{jcastillejoslopez@impan.pl}
\author[S.\ Evington]{Samuel Evington}
\address{\hskip-\parindent Samuel Evington, Institute of Mathematics, Polish Academy of Sciences, ul. {\'S}niadeckich 8, 00-656 Warszawa, Poland}
\curraddr{Mathematical Institute, University of Oxford, Radcliffe Observatory Quarter, Woodstock Road, Oxford, OX2 6GG}
\email{Samuel.Evington@maths.ox.ac.uk}
\author[A.\ Tikuisis]{Aaron Tikuisis}
\address{\hskip-\parindent Aaron Tikuisis, Department of Mathematics and Statistics, University of Ottawa, Ottawa, K1N 6N5, Canada.}
\email{aaron.tikuisis@uottawa.ca}
\author[S.\ White]{Stuart White}
\address{\hskip-\parindent Stuart White, School of Mathematics and Statistics, University of Glasgow, Glasgow, G12 8QW, Scotland and Mathematisches Institut der WWU M\"unster, Einsteinstra\ss{}e 62, 48149 M\"unster, Germany.}
\curraddr{Mathematical Institute, University of Oxford, Radcliffe Observatory Quarter, Woodstock Road, Oxford, OX2 6GG}
\email{stuart.white@maths.ox.ac.uk}
\author[W.\ Winter]{Wilhelm Winter}
\address{\hskip-\parindent
Wilhelm Winter, Mathematisches Institut der WWU M\"unster, Einsteinstra\ss{}e 62, 48149 M\"unster, Germany.}
\email{wwinter@uni-muenster.de}
\thanks{Research partially supported by: EPSRC (EP/N00874X/1 (AT), EP/R025061/1 (SE, SW), EP/K503058/1 and EP/J500434 (SE)), DFG (SFB 878, EXC 2044 Mathematics M\"unster: Dynamics -- Geometry -- Structure (WW)), NSERC (AT), NCN (2014/14/E/ST1/00525 (SE)), by an Alexander von Humboldt Foundation fellowship (SW), and a Methusalem grant of the Flemish Government (JC)}

\begin{document}
\maketitle

\begin{abstract}
We compute the nuclear dimension of separable, simple, unital, nuclear, $\mathcal Z$-stable $\mathrm{C}^*$-algebras. This makes classification accessible from $\mathcal Z$-stability and in particular brings large classes of $\mathrm{C}^*$-algebras associated to free and minimal actions of amenable groups on finite dimensional spaces within the scope of the Elliott classification programme.
\end{abstract}

\renewcommand*{\thetheorem}{\Alph{theorem}}

\section*{Introduction}

\noindent 
Nuclear dimension is a non-commutative generalisation of topological covering dimension to $\mathrm{C}^*$-algebras, introduced in \cite{WZ10}. A unital abelian $\mathrm{C}^*$-algebra consists of continuous functions on a compact Hausdorff space $X$; in this case the nuclear dimension recaptures the dimension of $X$.  At the other extreme lie simple $\mathrm{C}^*$-algebras, where nuclear dimension divides the exotic examples of \cite{Vi99,Ro03,To08,GK14}  from those accessible to $\mathrm{K}$-theoretic classification.  Indeed, through the work of generations of researchers (\cite{Ki95,Phi00,Go15,EGLN15,TWW17,Wi14} building on numerous works going back to \cite{El76}), we now have a complete classification of separable, simple, unital $\mathrm{C}^*$-algebras of finite nuclear dimension satisfying Rosenberg and Schochet's universal coefficient theorem (UCT) \cite{RS:UCT}. This provides the $\mathrm{C}^*$-analogue of the celebrated classification of amenable factors due to Connes and Haagerup \cite{Co76,Ha87}. 

A major task at this point is to identify simple nuclear $\mathrm{C}^*$-algebras of finite nuclear dimension, through establishing the Toms--Winter conjecture (\cite{ET08,Wi10}; the precise statement is \cite[Conjecture 9.3]{WZ10}, and an overview is given in \cite[Section 5]{Wi18}). This predicts that finite nuclear dimension is but one facet of a meta-notion of regularity for simple nuclear $\mathrm{C}^*$-algebras, with alternate descriptions of very different natures. Classification will then be accessed through regularity, the particular form depending on the example of interest.

In this paper we show that for separable, simple, unital, and nuclear $\mathrm{C}^*$-algebras, finite nuclear dimension is entailed by the tensorial absorption condition of $\mathcal Z$-stability, where $\mathcal Z$ is the Jiang--Su algebra of \cite{JS99} ($\mathcal Z$-stability will be described further below). Combining this with the main result of \cite{Wi12} gives the following theorem, which was predicted by the Toms--Winter conjecture.
Both implications making up this equivalence have striking applications: (i)$\Rightarrow$(ii) allows classification to be accessed via localisation at strongly self-absorbing algebras (\cite{Wi14}), while the implication (ii)$\Rightarrow$(i) proven here brings large classes of examples coming from topological dynamics within the scope of classification (see Corollaries \ref{CorE} and \ref{NewCor} below).

\begin{introtheorem}\label{ThmA}
Let $A$ be an infinite dimensional, separable, simple, unital, nuclear $\mathrm{C}^*$-algebra.
Then the following statements are equivalent:
\begin{enumerate}
\item[\rm{(i)}] $A$ has finite nuclear dimension;
\item[\rm{(ii)}]  $A$ is $\mathcal Z$-stable, i.e., $A\cong A\otimes\mathcal Z$.
\end{enumerate}
\end{introtheorem}

\smallskip

Let us put conditions (i) and (ii) into context.  
Tensorial absorption (or stability) phenomena are ubiquitous in operator algebras, originating with the characterisation of properly infinite von Neumann algebras $\mathcal M$ as those which tensorially absorb $\mathcal B(\mathcal H)$, i.e., are isomorphic to the von Neumann tensor product $\mathcal M \, \overline{\otimes} \, \mathcal B(\mathcal H)$. A major step in Connes' seminal work \cite{Co76} was to show that injective II$_1$ factors are McDuff, i.e., they absorb the hyperfinite II$_1$ factor tensorially, while Kirchberg famously characterised pure infiniteness for simple nuclear $\mathrm{C}^*$-algebras through tensorial absorption of the Cuntz algebra $\mathcal O_\infty$ (\cite{Ki95}). 

In a precise sense, the Jiang--Su algebra $\mathcal Z$ is the smallest possible, and so from this perspective the most natural, $\mathrm{C}^*$-analogue of the hyperfinite II$_1$ factor $\mathcal R$. It satisfies $\mathcal Z\otimes\mathcal Z\cong \mathcal Z$; moreover, this isomorphism occurs in a particularly strong way, analogous to the corresponding statement for $\mathcal R$ (see \cite{TW07}), and $\mathcal Z$ is the minimal non-trivial $\mathrm{C}^*$-algebra with this property (\cite{Wi11}). Accordingly, $\mathcal Z$-stability is the weakest non-trivial form of tensorial absorption.  Moreover, $\mathcal Z$-stability is to a vast extent necessary for $\mathrm{K}$-theoretic classification: $\mathcal Z$ has the same $\mathrm{K}$-theory as $\mathbb C$ --- these algebras are $\mathrm{KK}$-equivalent (think of this as a very weak kind of homotopy equivalence) --- and a unique trace. Therefore, under natural restrictions, a simple $\mathrm{C}^*$-algebra and its $\mathcal Z$-stabilisation share the same Elliott invariant ($\mathrm{K}$-theory and traces).

While all known constructions of the Jiang--Su algebra are a little involved, $\mathcal Z$-stability can be phrased entirely without reference to the algebra $\mathcal Z$ itself.  Reminiscent of McDuff's characterisation of $\mathcal{R}$-stability of a II$_1$ factor in terms of approximately central matrix algebras (\cite{McD70}), $\mathcal Z$-stability is equivalent to the existence of approximately central matrix cones which are large in a suitable sense (\cite{RW10}). Moreover, for simple, nuclear $\mathrm{C}^*$-algebras, via the groundbreaking work of Matui and Sato (\cite{MS12}), this largeness can be measured in trace (\cite{HO13}). This criterion has enabled $\mathcal Z$-stability to be established in a number of settings where a direct proof of finite nuclear dimension is unavailable (\cite{EN17,K17,CJKMST-D,KS18}).

Turning to (i) in Theorem \ref{ThmA}, one begins by noting that, through partitions of unity, a finite open cover of a compact Hausdorff space $X$ gives rise to a finite dimensional approximation of $C(X)$. These provide an explicit realisation of the completely positive approximations of Choi--Effros and Kirchberg (\cite{CE78,Ki77}) which characterise nuclearity for $\mathrm{C}^*$-algebras.  Covering dimension is determined by the existence of arbitrarily fine open covers which can be coloured so that no two sets of the same colour overlap; the number of colours needed determines the dimension.\footnote{The dimension is at most  $n$ when $n+1$ colours suffice. For example, the interval $[0,1]$ is a $1$-dimensional space, and has $2$-colourable arbitrarily fine open covers.} A finite colouring of an open cover gives a decomposition of the associated finite dimensional model of $C(X)$. In the general setting, completely positive approximations can be viewed as non-commutative partitions of unity.  The nuclear dimension (and its forerunner, the \emph{decomposition rank} from \cite{KW04}) is defined in terms of uniformly decomposable completely positive approximations; i.e., finitely colourable non-commutative partitions of unity. We recall the precise definition in Paragraph \ref{SubSecNDim}.

In principle, a system of completely positive approximations fully encodes a nuclear $\mathrm{C}^*$-algebra $A$, but extracting information is not straightforward without imposing additional structure on the approximations.\footnote{Indeed, as described further below, we use a  refined form of the completely positive approximation property which can be used to simultaneously approximate the trace simplex of $A$ in the proof of our main result.} With finite nuclear dimension, one can use the uniform bound on the number of colours to transfer properties of the finite dimensional algebras in the approximation back to $A$, albeit at the cost of suitably loosening these properties, often in a dimensional way (see \cite{Rob11, Wi10, Wi12}).  This strategy is at the heart of the construction of large approximately central matrix cones from nuclear dimension approximations in the proof of Theorem \ref{ThmA} (i)$\Rightarrow$(ii) given in \cite{Wi12}.

\medskip

The main result of this paper is the implication (ii)$\Rightarrow$(i) in Theorem \ref{ThmA}. We also determine the relationship between finite nuclear dimension and decomposition rank for simple $\mathrm{C}^*$-algebras in terms of quasidiagonality (see \cite{V93,Po,BK97,B06,KW04,MS14,TWW17} for the history, and partial solutions to this problem). 

\begin{introtheorem}
\label{thm:MainThm1}
Let $A$ be a  separable, simple, unital, nuclear, $\mathcal Z$-stable $\mathrm{C}^*$-algebra. Then $A$ has nuclear dimension  at most $1$.  If $A$ is also finite and all traces on $A$ are quasidiagonal, then $A$ has decomposition rank  at most $1$.
\end{introtheorem}

As $\mathrm{C}^*$-algebras whose nuclear dimension or decomposition rank is zero are necessarily approximately finite dimensional (AF for short) by \cite{KW04, WZ10}, the previous theorem determines the nuclear dimension and decomposition rank of simple unital $\mathrm{C}^*$-algebras.

\begin{introcorollary}\label{CorC}
The possible values of the nuclear dimension and decomposition rank of a simple, unital $\mathrm{C}^*$-algebra are $0$, $1$ and $\infty$. The value $0$ occurs precisely for AF algebras.
\end{introcorollary}

When $A$ has no traces in Theorem \ref{thm:MainThm1}, it is purely infinite, and the result is \cite[Theorem G]{BBSTWW} (completing the line of research in \cite{WZ10,MS14,RSS15,EN15}). With traces, partial results were first found in the unique trace case (\cite{MS14,SWW15}) and then when the trace space of $A$ is a Bauer simplex (\cite{BBSTWW}), i.e., has compact extreme boundary. The major technical hurdle which must be overcome to go beyond this setting is a strategy for combining behaviour at individual traces to obtain global properties uniformly over the trace space while respecting the affine structure.\footnote{In the Bauer simplex setting, the affine structure is not so important; one can just work with continuous functions on the boundary, as these extend to affine functions on the simplex.} We resolve this by introducing a completely new technique --- \emph{complemented partitions of unity} --- enabling us to perform partition of unity arguments inside a $\mathcal Z$-stable nuclear $\mathrm{C}^*$-algebra.  We discuss this in more detail below, but we first turn to the consequences of Theorem \ref{ThmA} for classification.

Combining the very recent stably finite results (\cite{Go15,EGLN15} and \cite[Corollary D]{TWW17}) with the celebrated Kirchberg--Phillips classification theorem,  simple, separable, unital, nuclear $\mathrm{C}^*$-algebras of finite nuclear dimension with the UCT  are classified by their Elliott invariants.\footnote{We will refer to this as `the classification theorem' below.} A key consequence of Theorem \ref{ThmA} is that the classification theorem can now be accessed using $\mathcal Z$-stability in place of finite nuclear dimension.

\begin{introcorollary}\label{CorD}
\label{cor:Classification}
Separable, simple, unital, nuclear, $\mathcal Z$-stable $\mathrm{C}^*$-algebras in the UCT class are classified by their Elliott invariants.
\end{introcorollary}

Even if $A$ is not itself $\mathcal Z$-stable, this can be forced by tensoring with $\mathcal Z$, at the cost of changing the algebra. So, in the language of \cite{Wi14}, separable, simple, unital, nuclear $\mathrm{C}^*$-algebras in the UCT class are classified up to $\mathcal Z$-stability.  

Recently, new techniques have been introduced to verify $\mathcal Z$-stability for $\mathrm{C}^*$-algebras constructed out of infinite dimensional objects where a direct nuclear dimension computation seems out of reach. For example, Elliott and Niu show that crossed products associated to minimal $\mathbb Z$-actions of mean dimension zero (in the sense of Gromov,  Lindenstrauss, and Weiss) on infinite dimensional spaces are $\mathcal Z$-stable, \cite{EN17}.  In contrast, direct estimates of the nuclear dimension, using Rokhlin dimension or similar techniques, only work for finite dimensional spaces (\cite{HWZ15,Sz15}). 

Even more recently, the new tool of \emph{almost finiteness} --- which can be viewed as a dynamical analogue of the combination of $\mathcal Z$-stability and amenability --- provides a method for obtaining $\mathcal Z$-stability of crossed products associated to free minimal actions of amenable groups (\cite{K17}).  This approach constructs approximately central matrix cones using tiling methods from \cite{DHZ15,CJKMST-D}, very much in the spirit of the Ornstein--Weiss Rokhlin lemma in the measurable setting.  In \cite{KS18}, almost finiteness is obtained for all free actions on finite dimensional spaces of a countable discrete group whose finite subgroups have subexponential growth. This covers groups of intermediate growth, such as the Grigorchuk group \cite{Gri80}. In contrast, direct nuclear dimension computations via Rokhlin dimension methods seem to require serious conditions on the group, such as finite asymptotic dimension (\cite{SWZ}).  

\begin{introcorollary}\label{CorE}
Let $G$ be a countably infinite, discrete group such that all finitely generated subgroups have subexponential growth.  Then crossed products $C(X)\rtimes G$ associated to free, minimal actions of $G$ on a compact metrisable finite dimensional space $X$ have finite nuclear dimension.  All these crossed products are covered by the classification theorem.
\end{introcorollary}

Additionally, \cite{CJKMST-D} shows that a generic free minimal action of a fixed amenable group $G$ on a Cantor set has $\mathcal Z$-stable crossed product. Accordingly, such crossed products are generically classifiable.

\begin{introcorollary}\label{NewCor}
Let $G$ be a countably infinite, discrete, amenable group, and let $X$ be the Cantor set. Then the set of all free, minimal actions $\alpha:G\curvearrowright X$, for which the crossed product $C(X)\rtimes_\alpha G$ has finite nuclear dimension, and so is covered by the classification theorem, is comeagre in the set of all free minimal actions of $G$ on $X$.\footnote{The Polish space structure on the space of all such free minimal actions is described just before \cite[Theorem 4.2]{CJKMST-D}.}
  \end{introcorollary}
  
In the remainder of the introduction, we change gears and discuss the proof of our main results, isolating the key new technique --- complemented partitions of unity --- introduced in this paper.  The reader may wish to return to the following two sections in parallel with the main body of the paper.

\section*{\sc From $\mathcal Z$-stability to finite nuclear dimension} 

\noindent
For most of the rest of the introduction, we think of $A$ as a separable, simple, unital, nuclear $\mathrm{C}^*$-algebra with at least one trace.  At times we will allow more general $A$, and make this explicit.

The passage from $\mathcal Z$-stability to finite nuclear dimension runs through the classification of order zero maps (equivalently $^*$-homomorphisms out of cones) by traces. As usual, there are two required  components: (i) existence, and (ii) uniqueness (up to approximate unitary equivalence). Somewhat more precisely, one should:
\begin{enumerate}
\item Produce a sequence of approximately order zero maps $A\rightarrow A\otimes\mathcal Z$ of nuclear dimension zero which uniformly approximates $\id_A\otimes 1_{\mathcal Z}:A\rightarrow A\otimes\mathcal Z$ on traces.
\item Establish a uniqueness theorem for approximately order zero maps which uniformly approximate the two maps in \eqref{Intro.E1} below on traces.
\end{enumerate}

Using the fact that $\mathcal Z$ is strongly self-absorbing, to show that $A\otimes\mathcal Z$ has nuclear dimension at most $1$ it suffices to control the nuclear dimension of the map $\id_A\otimes 1_{\mathcal Z}:A\rightarrow A\otimes\mathcal Z$. By taking a positive contraction
 $h\in\mathcal Z$ with spectrum $[0,1]$, $\id_A\otimes 1_{\mathcal Z}$ splits as the sum of the two order zero maps
\begin{equation}\label{Intro.E1}
\id_A\otimes h:A\rightarrow A\otimes \mathcal Z\text{ and }\id_A\otimes (1_{\mathcal Z}-h):A\rightarrow A\otimes\mathcal Z.
\end{equation}
Now (i) and (ii) can be used to show that each map in (\ref{Intro.E1}) has nuclear dimension zero, and so $\id_A\otimes 1_{\mathcal Z}$, and hence $A\otimes\mathcal Z$ has nuclear dimension at most $1$. To obtain (i) and (ii), we introduce the new technique of complemented partitions of unity. For ease of exposition, we will primarily focus on the existence component (i) in the discussion below.

With hindsight, the strategy above was initiated in \cite{MS14} in the case when $A$ is quasidiagonal and has a unique trace, when the existence component (i) is an immediate consequence of the quasidiagonality assumption.  An `order-zero quasidiagonality' result was subsequently developed in \cite{SWW15} to handle the case where $A$ is not assumed to be quasidiagonal (but still has unique trace).  When $A$ has many traces, these methods give approximately order zero maps of nuclear dimension zero each behaving well on an individual trace; the challenge is to combine them into a single map with the required behaviour on all traces simultaneously. Via compactness of the tracial state space, $T(A)$, one can find an open cover $U_1,\dots,U_k$ of $T(A)$, and approximately order zero maps $\theta_1,\dots,\theta_k:A\rightarrow A\otimes\mathcal Z$ of nuclear dimension zero so that $\theta_i$ models $\id_A\otimes 1_{\mathcal Z}$ on the traces in $U_i$. Given pairwise orthogonal approximately central positive contractions $(e_i)_{i=1}^k$ in $A\otimes\mathcal Z$, one can define a map $\theta:A\rightarrow A\otimes\mathcal Z$ by
\begin{equation}\label{Intro.E2}
\theta(\cdot)=\ssum_{i=1}^k e_i^{1/2}\theta_i(\cdot )e_i^{1/2};
\end{equation}
properties of the $e_i$ combine with those of the $\theta_i$ to ensure that $\theta$ is approximately order zero and of nuclear dimension zero.   The idea is that $\theta$ is obtained by `gluing together' the $\theta_i$ over the $e_i$, which should be well chosen with respect to the open cover $(U_i)_{i=1}^k$ so $\theta$ has the global tracial behaviour required by (i). To do this it is necessary to control $\tau(e_i\theta_i(\cdot))$ for $\tau\in T(A)$.

This was achieved in \cite{BBSTWW} when $T(A)$ is a Bauer simplex, i.e., the extremal boundary $\partial_eT(A)$ is compact.  In this case, continuous affine functions on $T(A)$ are in one-to-one correspondence with the continuous functions on the extreme boundary, and we only need to use extremal traces in the compactness argument to obtain the $U_i$'s. The required $(e_i)_{i=1}^k$ arise via Ozawa's theory of $\mathrm{W}^*$-bundles (\cite{Oz13}): one takes a partition of unity $(f_i)_{i=1}^k$ in $C(\partial_eT(A))$ subordinate to the cover $(U_i)_{i=1}^k$ and then one uses triviality of the $\mathrm{W}^*$-bundle associated to $A\otimes\mathcal Z$ to produce the elements $e_i$ so that $\tau(e_i)\approx f_i(\tau)$.  One can then show (see (\ref{Intro.E4}) below) that
\begin{equation}\label{Intro.E3}
\tau(e_ia)\approx \tau(e_i)\tau(a)\approx f_i(\tau)\tau(a),\quad a\in A,\ \tau\in\partial_eT(A),
\end{equation}
and then a partition of unity calculation shows that the map in (\ref{Intro.E2}) has the properties required by (i). 

The uniqueness ingredient (ii) was obtained in \cite{BBSTWW} under the same compactness assumption on the tracial extreme boundary of $A$.\footnote{A sketch of the argument can be found in the introduction to \cite{BBSTWW}.} This uses a detailed analysis of the structure of relative commutants, building on Matui and Sato's groundbreaking techniques for transferring structure between von Neumann and $\mathrm{C}^*$-algebras (\cite{MS12,MS14}).  As with (i), a local-to-global argument at the level of traces is a crucial ingredient. Without getting into detail, this again boils down to the existence of approximately central pairwise orthogonal elements $(e_i)_{i=1}^k$ satisfying (\ref{Intro.E3}) for a suitable partition of unity $(f_i)_{i=1}^k$ of $\partial_eT(A)$.

\smallskip

Fundamental difficulties arise outside the setting of Bauer simplices. Not all continuous functions on $\partial_eT(A)$ extend to continuous affine functions on $T(A)$, and $C(\partial_e T(A))$ need not embed into the centre of the strict closure of $A$ (the starting point for producing $(e_i)_{i=1}^k$ in the Bauer simplex case).  Worse, this centre can even be trivial. Moreover, to perform the required gluing argument it is not enough to specify the affine function $\hat{e}_i:\tau\mapsto \tau(e_i)$ that $e_i$ should induce.  Rather one needs to control $\widehat{e_ia}:\tau\mapsto \tau(e_ia)$ for a suitable finite collection of elements $a\in A$; cf.\ (\ref{Intro.E3}). To address this, one has to understand the subtle question of how $\mathrm{C}^*$-algebra multiplication interacts with the affine functions on $T(A)$ induced by positive elements.  

For an approximately central sequence $e_i=(e_{i,n})_{n=1}^\infty$ in $A$ and $a\in A$, 
\begin{equation}\label{Intro.E4}
\lim_{n\rightarrow\infty}|\tau(e_{i,n}a)-\tau(a)\tau(e_{i,n})|=0,\quad \tau\in\partial_eT(A),
\end{equation}
with uniform convergence when $\partial_eT(A)$ is compact.\footnote{See \cite[Lemma 4.2(i)]{Sa12} and \cite[Proposition 4.3.6]{Ev} for closely related statements. This point is developed further in \cite{CETWW}.} This is why (\ref{Intro.E3}) follows from $\tau(e_i)\approx f_i(\tau)$ in the Bauer simplex case. But without compactness, uniform convergence fails, even in straightforward cases (see \cite[Example 3.3]{CETWW}). In general, one cannot obtain $\widehat{e_ia}$ automatically from $\hat{e}_i$ outside the Bauer setting.

The key novelty of this paper is a technique for producing the $(e_i)_{i=1}^k$ required for the gluing procedure of (\ref{Intro.E2}) taking into account the affine structure of $T(A)$.  The precise definition is given in Definition \ref{defn:CPOU}; for now, we give a relatively informal version of the characterisation from Proposition \ref{prop:CPoUReformulation}.
\begin{introdefinition}\label{intro:defCPoU}
Let $A$ be a separable, unital $\mathrm{C}^*$-algebra with $T(A)$ compact and non-empty. Then $A$ has \emph{complemented partitions of unity} (CPoU) if, whenever  $a_1,\dots,a_k$ are positive elements in $A$ and $\delta>0$ is such that
\begin{equation}\label{Intro.E5}
\sup_{\tau\in T(A)}\min_i\tau(a_i)<\delta,
\end{equation}
 there exist pairwise orthogonal, approximately central, positive contractions $e_1,\dots,e_k\in A$ such that 
\begin{enumerate}[(a)]
\item $\tau(\sum_{i=1}^k e_i)\approx 1$, for all $\tau\in T(A)$;
\item $\tau(a_ie_i)\lessapprox\delta \tau(e_i)$ for all $\tau\in T(A)$ and $i=1,\dots,k$.
\end{enumerate}
\end{introdefinition}

Here (\ref{Intro.E5}) means that $(\{\tau\in T(A):\tau(a_i)<\delta\})_{i=1}^k$ is an open cover of $T(A)$. The conclusion (a) says that the $\hat{e}_i$ form an approximate partition of unity. Normally partitions of unity are constructed subordinate to a specified open cover or family of functions. But we need to control the affine functions $\widehat{e_ia_i}$ not just $\hat{e}_i$. This is the role of condition (b). Note that it has the effect of ensuring that $\hat{e}_i$ is to some extent subordinate to the complement $1-\hat{a}_i$ of $\hat{a}_i$: if $\tau\in T(A)$ satisfies $\tau(a_i)\approx 1$, then $ \tau(e_i)\approx \tau(a_ie_i)\lessapprox \delta\tau(e_i)$, so that $\tau(e_i)$ must be small.\footnote{This intuition is the reason behind our choice of terminology.}

In applications, the $a_i$ in (\ref{Intro.E5}) will arise from compactness of the tracial state space. Returning to the sketch proof of existence above, we could set $a_i:=1_{A\otimes\mathcal Z}-\theta_i(1_A)$, so that if $\theta_i$ models $\id_A\otimes 1_{\mathcal Z}$ on $U_i$ (on $1_A$ up to $\delta$), we obtain (\ref{Intro.E5}). Then given $(e_i)_{i=1}^k$ in $A\otimes\mathcal Z$ as in Definition \ref{intro:defCPoU} for these $a_i$, defining $\theta$ as in (\ref{Intro.E2}), we would have 
\begin{equation}
1-\tau(\theta(1_A))\stackrel{(\rm{a})}\approx \ssum_{i=1}^k\tau(e_ia_i)\stackrel{(\rm{b})}\lessapprox \delta\ssum_{i=1}^k\tau(e_i)\stackrel{(\rm{a})}\approx\delta,\quad \tau\in T(A),
\end{equation}
i.e., $\theta(1_A)$ globally models $1_A\otimes1_{\mathcal Z}$ in trace.  With a bit more care, the same strategy can be used to produce a $\theta$ which globally models $\id_A\otimes 1_{\mathcal Z}$ on an arbitrary finite subset of $A$ in place of $1_A$, establishing the required existence component (i). 

The uniqueness ingredient (ii) is also obtained by using complemented partitions of unity in place of the $\mathrm{W}^*$-bundle methods used in \cite{BBSTWW} for the special case of Bauer simplices. This gives the following result which we set out in Sections \ref{S4} and \ref{S5}.

\begin{introtheorem}\label{PropG}
Let $A$ be a separable, simple, unital, nuclear, $\mathcal Z$-stable $\mathrm{C}^*$-algebra with $T(A)$ compact and non-empty, and with complemented partitions of unity. Then $\dimnuc(A)\leq 1$. If additionally all traces on $A$ are quasidiagonal, then $\mathrm{dr}(A)\leq 1$.
\end{introtheorem}

Having isolated the abstract condition of complemented partitions of unity, the other main challenge of this paper is to verify this condition in the presence of $\mathcal Z$-stability.
This is the following theorem, which together with Theorem \ref{PropG} gives Theorems \ref{ThmA} and \ref{thm:MainThm1} and their corollaries.
\begin{introtheorem}\label{thm:mainCPoU}
Let $A$ be a separable, nuclear, $\mathcal Z$-stable $\mathrm{C}^*$-algebra with $T(A)$ compact and non-empty. Then $A$ has complemented partitions of unity.
\end{introtheorem}

Note that we establish CPoU in Theorem \ref{thm:mainCPoU} without assuming $A$ to be unital. In the main results of our paper, unitality enters through Theorem \ref{PropG}, as the methods from \cite{BBSTWW} are only available in the presence of a unit. However, in \cite{CE} the first two authors develop the necessary machinery to extend Theorem \ref{PropG}, and hence the main results of this paper to the 
non-unital setting. It should be noted that their result, although more general, 
requires Theorem \ref{thm:mainCPoU} and the CPoU technology introduced in this paper.

The partitions of unity of Theorem \ref{thm:mainCPoU} are also applicable in other situations, some of which are discussed in \cite{CETWW} (for example the equivalence of strict comparison and $\mathcal Z$-stability under the assumption of uniform property $\Gamma$).   Further, complemented partitions of unity will play a crucial role in the abstract approach to classification of $\mathrm{C^*}$-algebras being developed in \cite{CGSTW}.

\section*{\sc Obtaining complemented partitions of unity}

\noindent
We end the introduction with an outline of the proof of Theorem \ref{thm:mainCPoU}.  Fix $a_1,\dots,a_k$ and $\delta>0$ as in Definition \ref{intro:defCPoU}.  The three main components of the proof are as follows:
\begin{enumerate}[(1)]
\item \emph{Non-orthogonal partitions of unity}: We produce approximately central positive contractions $e_1,\dots,e_k$ satisfying (a) and (b), but without the pairwise orthogonality condition.  \label{Intro-Sect3.1}
\item \emph{Tracial projectionisation}: We replace the elements $(e_i)_{i=1}^k$ by elements $(p_i)_{i=1}^k$ which are approximately projections in trace while retaining conditions (a) and (b).\label{Intro-Sect3.2}
\item \emph{Orthogonalisation}: We replace $(p_i)_{i=1}^k$ by pairwise orthogonal tracial approximate projections, still satisfying (b) but at the cost of replacing  (a) with the weaker condition $\tau(\sum_{i=1}^k p_i)\approx 1/k$.\label{Intro-Sect3.3}
\end{enumerate}

These steps produce elements satisfying (b) but they only cover approximately $1/k$ of the total trace.  So we repeat the procedure\footnote{This in fact requires a stronger version of step (\ref{Intro-Sect3.1}).} underneath the tracially approximate projection $1_A-\sum_{i=1}^kp_i$. Stage (\ref{Intro-Sect3.2}) is necessary at this point so that we have a (tracially approximate) orthogonal complement.  Theorem \ref{thm:mainCPoU} then follows using a maximality argument.

\smallskip
Our proof of step (\ref{Intro-Sect3.1}) uses nuclearity through the refined version of the completely positive approximation property obtained in \cite{BCW16}. This gives approximations of the identity map of the form
\begin{equation}\label{Intro.E6}
\xymatrix{A\ar[rr]\ar[dr]_{\psi}&&A\\&F\ar[ur]_\phi},
\end{equation}
where $F$ is finite dimensional, $\psi$ is completely positive, contractive and approximately order zero, and $\phi$ is a convex combination of order zero maps.  In particular, $\phi$ preserves traces, and $\psi$ approximately does so.  This gives an approximation of $T(A)$ by $T(F)$ --- a finite dimensional simplex. The idea is to build complemented partitions of unity for $\psi(a_1),\dots,\psi(a_k)$ in $F$ (by taking the corresponding $e_i$ to be supported only on those full matrix summands of $F$ where $\psi(a_i)$ is small) and then push this back into $A$ (at the cost of losing orthogonality).

Stages (\ref{Intro-Sect3.2}) and (\ref{Intro-Sect3.3}) use the extra space given by a tensor factor of $\mathcal Z$ (and do not need nuclearity). To get a feeling for (\ref{Intro-Sect3.3}), suppose $A\cong A\otimes\mathcal Q$, where $\mathcal Q$ is the universal UHF algebra. One can identify $A\cong A\otimes M_k$, so that $a_i$ and $p_i$ are approximately of the form $\tilde{a}_i\otimes 1_{M_k}$ and $\tilde{p}_i\otimes 1_{M_k}$ respectively. Letting $(d_{ij})$ be a system of matrix units for $M_k$, one defines new approximate projections by $\tilde{p}_i\otimes d_{ii}$, which have approximately $\frac{1}{k}$ of the trace of $p_i$.   Likewise, for (\ref{Intro-Sect3.2}), given elements $e_i\in A$, approximately of the form $\tilde{e}_i\otimes 1_k$ in $A\otimes M_k$, we can use the additional space to redistribute the trace of $\tilde{e}_i$ over the diagonal elements of $M_k$ to obtain a $\frac{1}{k}$-approximate projection in trace.

In fact, neither (\ref{Intro-Sect3.2}) nor (\ref{Intro-Sect3.3}) require the full strength of $\mathcal Z$-stability; all that is needed is tracial divisibility of positive elements in an approximately central fashion reminiscent of Murray and von Neumann's property $\Gamma$ for II$_1$ factors. We introduce the concept of \emph{uniform property $\Gamma$} for $\mathrm{C}^*$-algebras in Section \ref{section:gamma}, and use it to obtain complemented partitions of unity in Section \ref{S3}.  The follow-up paper \cite{CETWW} further develops the theory of uniform property $\Gamma$ and complemented partitions of unity, giving other applications to the Toms--Winter conjecture.

\subsection*{Acknowledgments}

We thank Christopher Schafhauser and G\'abor Szab\'o for helpful comments on a preliminary version of this paper.

\numberwithin{theorem}{section}
\section{Preliminaries}

\noindent
For a $\mathrm{C}^*$-algebra $A$, we write $A_+$ for the set of positive elements, $A^1$ for the closed unit ball, and $A_+^1$ for the set of positive contractions. 

\subsection{Nuclear dimension and decomposition rank}\label{SubSecNDim} Let $A$ and $B$ be $\mathrm{C^*}$-algebras. A completely positive map $\phi:A\rightarrow B$ has \emph{order zero} if it preserves orthogonality, i.e., $\phi(x)\phi(y)=0$ whenever $x,y\in A_+$ satisfy $xy=0$. Recall from \cite{WZ10}, that $A$ has \emph{nuclear dimension} at most $n$, written $\dimnuc(A)\leq n$, if there exists a net $(F_i,\psi_i,\phi_i)_i$, where each $F_i$ is a finite dimensional $\mathrm{C}^*$-algebra, $\psi_i:A\rightarrow F_i$ is a completely positive and contractive (c.p.c.) map, $\phi_i:F\rightarrow A$ is a completely positive map which is a sum of at most $n+1$ c.p.c.\ maps of order zero, and $\|\phi_i(\psi_i(a))-a\|\stackrel{i}\rightarrow 0$ for all $a\in A$.
Note that $\|\phi_i\|\leq n+1$.
The stronger (and historically earlier) condition from \cite{KW04} that $A$ has \emph{decomposition rank} at most $n$ is satisfied when in addition, each $\phi_i$ can be taken to be contractive. Recall too that these definitions can be made at the level of maps; see \cite[Definition 2.2]{TikW}.

\subsection{Traces} In this paper a \emph{trace} on a $\mathrm{C}^*$-algebra $A$ means a tracial state, and we denote the collection of all traces by $T(A)$. The set of traces $T(A)$ is equipped with the weak$^*$-topology it inherits from $A^*$, and if $X \subseteq T(A)$ then $\overline{X}$ refers to the closure in this topology.  For a non-empty set $X \subseteq T(A)$, define the seminorm $\|\cdot\|_{2,X}$ on $A$ by
\begin{equation} \|a\|_{2,X} := \sup_{\tau \in X} (\tau(a^*a))^{1/2},\quad a\in A. \end{equation}
Note that $\|\cdot\|_{2,T(A)}$ is a norm if, for every non-zero $a\in A$, there exists $\tau\in T(A)$ with $\tau(a^*a) >  0$, and so in particular when $A$ is simple and $T(A)\neq\emptyset$.    Any $^*$-homomorphism $\phi:A \to B$ is $(\|\cdot\|_{2,T(A)},\|\cdot\|_{2,T(B)})$-contractive (since for $\tau \in T(B)$, $\tau\circ\phi$ is a contractive tracial functional).  The (semi)norm $\|\cdot\|_{2,T(A)}$ plays a notable role in recent structural results (\cite{MS12,KR14,Oz13,BBSTWW}), where it is often denoted $\|\cdot\|_{2,u}$.

\subsection{Ultrapowers}\label{S1.2} Throughout, $\omega$ will stand for a fixed free ultrafilter on $\mathbb N$. 
Given a separable $\mathrm{C}^*$-algebra $A$, the \emph{ultrapower} of $A$ is
\begin{equation}
A_\omega:=\ell^\infty(A)/\{(a_n)_{n=1}^\infty\in\ell^\infty(A):\lim_{n\to\omega}\|a_n\|=0\}.
\end{equation}
When $T(A)\neq \emptyset$, the \emph{uniform tracial ultrapower} is defined by 
\begin{equation}\label{def:uniformtracialultrapower} A^\omega := A_\omega / J_A, \end{equation}
where (abusing notation to use representative sequences in $\ell^\infty(A)$ to denote elements in $A_\omega$), 
\begin{equation} J_{A} := \big\{(a_n)_{n=1}^\infty \in A_\omega: \lim_{n\to\omega} \|a_n\|_{2,T(A)} = 0\big\}, \end{equation}
is the \emph{trace-kernel ideal}.   We will also use representative sequences to stand for elements in $A^\omega$.

Notice that $A$ embeds canonically into $A_\omega$ as constant sequences. We will often regard $A$ as a subalgebra of the ultrapower, for example to form the central sequence algebra $A_\omega\cap A'$. We shall adopt a similar notation in the case of uniform tracial ultrapowers.\footnote{\label{NormFoot}When $\|\cdot\|_{2,T(A)}$ is a norm, $A$ embeds canonically into $A^\omega$ as constant sequences, and we write $A^\omega \cap A'$ for the central sequence algebra. In general only a quotient of $A$ embeds in $A^\omega$; in this case we view $A^\omega$ as an $A$-bimodule, and continue to write $A^\omega\cap A'$ for the $A$-central elements of $A^\omega$.} 

We highlight two examples of the uniform tracial ultrapower.
\begin{enumerate}[(i)]
\item If $A$ has a unique trace $\tau$, then the GNS representation $\pi_\tau$ generates a finite von Neumann factor $\mathcal M:=\pi_\tau(A)''$, which is II$_1$ when it is infinite dimensional.  In this case, the Kaplansky density theorem canonically identifies $A^\omega$ with the II$_1$ factor ultraproduct $\mathcal M^\omega$.\label{TracialUltrapowerEx1}

\item When $A$ is unital, $T(A)$ forms a Choquet simplex. If additionally the extreme boundary $\partial_eT(A)$ is compact, then $A^\omega\cong \mathcal M^\omega$, where $\mathcal M$ is the $\mathrm{W}^*$-bundle obtained from the strict closure of $A$ introduced by Ozawa in \cite{Oz13}, and $\mathcal M^\omega$ is the $\mathrm{W}^*$-bundle ultraproduct defined in \cite[Section 3.2]{BBSTWW}.  This identification uses Kaplansky's density theorem in the strict-topology.\footnote{On $\|\cdot\|$-bounded subsets of $\mathcal M$, the strict topology agrees with the $\|\cdot\|_{2,T(A)}$-topology; see \cite[Proposition 3.2.15]{Ev}.} \label{TracialUltrapowerEx2}
\end{enumerate}

\subsection{Limit traces}
Any sequence of traces $(\tau_n)_{n=1}^\infty$ on $A$ induces a trace on $A_\omega$ by
\begin{equation} (a_n)_{n=1}^\infty \mapsto \lim_{n\to\omega} \tau_n(a_n). \end{equation}
A trace on $A_\omega$ of this form is called a \emph{limit trace}, and we will often denote a limit trace by a sequence $(\tau_n)_{n=1}^\infty$ that induces it. Evidently, every limit trace vanishes on $J_A$, and thereby also induces a trace on $A^\omega$.  We abuse notation, and write $T_\omega(A)$ for the collection of these limit traces on both $A_\omega$ and $A^\omega$; context will make it clear to which ultrapower such a trace applies. With this notation,
\begin{equation}
\label{eq:TraceKernelEq} J_{A} = \{x \in A_\omega: \|x\|_{2,T_\omega(A)} = 0\}. \end{equation}

\subsection{Kirchberg's $\epsilon$-test}
As with the norm ultrapower, the uniform tracial ultrapower lends itself to reindexing or diagonal sequence arguments which can be used to turn statements involving $\|\cdot\|_{2,T(A)}$ approximations in $A$ into exact statements in $A^\omega$.  Kirchberg's $\epsilon$-test (\cite[Lemma A.1]{Kir06}) provides a particularly convenient tool for this type of argument. A number of detailed examples of its application can be found in \cite[Section 1.3]{BBSTWW}; we give a further example here.

\begin{lemma}\label{BallComplete}
Let $A$ be a separable $\mathrm{C}^*$-algebra with $T(A)\neq\emptyset$. Then the unit ball of $A^\omega$ is $\|\cdot\|_{2,T_\omega(A)}$-complete.
\end{lemma}
\begin{proof}
Let $(x^{(k)})_{k=1}^\infty$ be a $\|\cdot\|_{2,T_\omega(A)}$-Cauchy sequence in $(A^\omega)^1$.
Take a representative sequence $(x^{(k)}_n)_{n=1}^\infty$ in $A^1$ for each $x^{(k)}$, and set
\begin{equation}
\epsilon^{(k)}:=\sup\big\{\|x^{(l)}-x^{(l')}\|_{2,T_\omega(A)}:l,l'\geq k\big\}.
\end{equation}
Define $f^{(k)}_n:A^1\rightarrow[0,2]$ by $f^{(k)}_n(y):=\max\{\|y-x^{(k)}_n\|_{2,T(A)}-\epsilon^{(k)},0\}$.

For $k_0\in\mathbb N$, we have $\lim_{n\rightarrow\omega}f^{(k)}_n(x^{(k_0)}_n)=0$, for $k=1,\dots,k_0$.  
Kirchberg's $\epsilon$-test (\cite[Lemma A.1]{Kir06}), gives a sequence $(x_n)_{n=1}^\infty$ in $A^1$ representing some $x\in A^\omega$ with $\lim_{n\rightarrow\omega}f^{(k)}_n(x_n)=0$ for all $k\in\mathbb N$. Therefore
\begin{equation}
\|x-x^{(k)}\|_{2,T_\omega(A)}\leq\epsilon^{(k)},\quad k\in\mathbb N.
\end{equation}
As $(x^{(k)})_{k=1}^\infty$ is $\|\cdot\|_{2,T_\omega(A)}$-Cauchy, we have $\epsilon^{(k)}\rightarrow 0$ and so $x^{(k)}\rightarrow x$.
\end{proof}

\subsection{Central surjectivity}
The following central surjectivity result is a key tool in transferring structural properties from $A^\omega$ to $A_\omega$ via Matui--Sato's property (SI).  It has its origins in Sato's work \cite{Sa11}, and was established in general by Kirchberg and R\o{}rdam.  Part (i) of the lemma below follows by combining \cite[Propositions 4.5(iii) and 4.6]{KR14}.\footnote{Note that while \cite[Proposition 4.6]{KR14} has a hypothesis that  $A$ is unital, for the proof the unit can be taken in the minimal unitisation.} 
\begin{lemma}[Central surjectivity]
\label{lem:CentralSurjection}
Let $A$ be a separable $\mathrm{C}^*$-algebra with $T(A)\neq \emptyset$.
Let $S \subset A_\omega$ be a separable subset and let $\bar{S} \subset A^\omega$ be its image.
Then
\begin{enumerate}
\item[\rm{(i)}]  the image of $A_\omega \cap S'$ in $A^\omega$ is $A^\omega \cap \bar{S}'$ (the commutant of $\bar{S}$ in $A^\omega$), and
\item[\rm{(ii)}]  for an $h\in (A_\omega\cap S')_+^1$ whose image $\bar{h}$ in $A^\omega$ is a projection, the image of $\{x\in A_\omega\cap S':hx=x\}$ is $\{y\in A^\omega\cap \bar{S}':\bar{h}y=y\}$.
\end{enumerate}
\end{lemma}
\begin{proof}[Proof of \rm{(ii)}]
Given $y\in A^\omega\cap \bar{S}'$ with $\bar{h}y=y$, lift $y$ to $\tilde{x}\in A_\omega\cap S'$ by (i).  
For each $n\in \mathbb N$, choose a positive contraction $f_n \in C_0((0,1])_+$ such that $f_n(1)=1$ and which is supported on $(1-1/n,1]$. Then as $\bar{h}$ is a projection, $f_n(h)\tilde{x}$ maps to $\bar{h}y=y\in A^\omega$ for each $n$. Moreover, we have $\|hf_n(h)-f_n(h)\| \leq 1/n$, and thus $\|hf_n(h)\tilde{x}-f_n(h)\tilde{x}\|\rightarrow 0$.  By Kirchberg's $\epsilon$-test (\cite[Lemma A.1]{Kir06}), there exists some $x\in A_\omega\cap S'$ with $hx=x$ which maps to $y\in A^\omega$.
\end{proof}

\subsection{Compactness of $T(A)$} Throughout the rest of the paper, our standard assumption will be that $T(A)$ is non-empty and compact. We end this section by recording two consequences of this.

\begin{lemma}\label{lem:ZeroInAomega:easy}
Let $A$ be a separable $\mathrm{C}^*$-algebra with $T(A)$ non-empty and compact. Let $\iota:A\rightarrow A^\omega$ denote the canonical $^*$-homomorphism.\footnote{Which is an embedding when $\|\cdot\|_{2,T(A)}$ is a norm.} If $x \in A^\omega$ satisfies $\iota(a)x=x\iota(a)=0$ for all $a\in A$ then $x=0$.
\end{lemma}
\begin{proof}
Assume that $x$ is a contraction, and for a contradiction that for some $\tau\in T_\omega(A)$ we have $\tau(x^*x)\geq\epsilon>0$.
Let $\tau=(\tau_n)_{n=1}^\infty$ with $\tau_n \in T(A)$ for all $n$. Since $\tau\circ\iota$, which a priori is a tracial functional on $A$ of norm at most $1$, is the limit as $n\rightarrow\omega$ of the tracial states $\tau_n$, by compactness of $T(A)$, $\tau\circ\iota$ is a tracial state on $A$. Therefore there exists $a\in A_+^1$ with $\tau(\iota(a))>1-\epsilon$. 
The hypothesis that $x$ is orthogonal to $A$ ensures that $\iota(a)+x^*x$ is a contraction. 
Thus,
\begin{equation} 1 \geq \tau(\iota(a)+x^*x) = \tau(\iota(a))+\tau(x^*x) > 1-\epsilon+\epsilon = 1, \end{equation}
a contradiction. Thus $\|x\|_{2,\tau}=0$ for all $\tau\in T_\omega(A)$, so $x=0$.
\end{proof}

\begin{proposition}
\label{lem:ZeroInAomega}
Let $A$ be a separable $\mathrm{C}^*$-algebra with $T(A)\neq \emptyset$. Then $A^\omega$ is unital if and only if $T(A)$ is compact. 
\end{proposition}
 
\begin{proof}
Suppose $T(A)$ is compact. Let $(e_n)_{n=1}^\infty$ be an approximate unit for $A$, and let $e=(e_n)_{n=1}^\infty \in (A^\omega)_+^1$, so that $ea=a$ for all $a \in A$.
Given any element $x\in A^\omega$, a standard application of Kirchberg's $\epsilon$-test gives $f\in (A^\omega)_+^1$ such that $fx=x$ and $fa=a$ for all $a\in A$ (see \cite[Lemma 1.16]{BBSTWW} with $S_1=S_2=\{0\}$). 
Then $(e-f)a=0$ for all $a \in A$, whence $e=f$ by Lemma \ref{lem:ZeroInAomega:easy}. 
It follows that $ex=fx=x$.
Thus, $e$ was already a unit for $A^\omega$.

Conversely, if $A^\omega$ is unital, represent $1_{A^\omega}$ by a sequence $(e_n)_{n=1}^\infty$ in $A_+^1$. 
Writing $\gamma_n:=\inf_{\tau\in T(A)} \tau(e_n)$, since $\tau(1_{A^\omega})=1$ for all $\tau \in T_\omega(A)$, we have $\lim_{n\to\omega}\gamma_n=1$. Now suppose that some net $(\tau_i)$ in $T(A)$ converges to a functional $\sigma$, so that $\sigma$ is a tracial functional and $\|\sigma\|\leq 1$. Then $\sigma(e_n)=\lim_i\tau_i(e_n)\geq \gamma_n \rightarrow 1$, so $\sigma$ is a state on $A$.
This shows that $T(A)$ is weak$^*$-closed in the (compact) unit ball of $A^*$.
\end{proof}

\section{Uniform property $\Gamma$ for $\mathrm{C}^*$-algebras}\label{section:gamma}

\noindent
Property $\Gamma$ for II$_1$ factors was introduced by Murray and von Neumann in \cite{MvN43} to prove the existence of non-hyperfinite II$_1$ factors. In modern language, a separably acting II$_1$ factor $\mathcal M$ has property $\Gamma$ if its central sequence algebra\footnote{Here $\mathcal M^\omega$ refers to the tracial von Neumann ultrapower of $\mathcal M$, which is again a II$_1$ factor with the induced trace $\tau_{\mathcal M^\omega}$ from $\mathcal M$.} $\mathcal M^\omega\cap \mathcal M'$ is non-trivial. In this case Dixmier showed that $\mathcal M^\omega\cap \mathcal M'$ is diffuse (\cite{Di69}), and hence for each $n\in\mathbb N$ one can find orthogonal projections $p_1,\dots,p_n\in\mathcal M^\omega\cap\mathcal M'$ with $\tau_{\mathcal M^\omega}(p_i)=1/n$. It is in this formulation that property $\Gamma$ has been most often used to obtain striking structural consequences; see \cite{Pi01,Chr01,CPSS03}.

In this section we introduce a uniform version of property $\Gamma$ for $\mathrm{C}^*$-algebras, motivated by Dixmier's formulation.\footnote{Versions of property $\Gamma$ for $\mathrm{C}^*$-algebras have been considered previously in connection with the similarity property \cite{GJS,QS16}.  For example, \cite{QS16} considers those $\mathrm{C}^*$-algebras all of whose II$_1$ factor representations have property $\Gamma$. The key difference between this concept and our notion is that uniform property $\Gamma$ requires division uniformly in all traces, as opposed to pointwise in trace.}  Recall our convention that $A^\omega$ denotes the \emph{uniform tracial ultrapower} from \eqref{def:uniformtracialultrapower}, which is unital when $T(A)$ is non-empty and compact, and $A^\omega\cap A'$ consists of the $A$-central elements of $A^\omega$ (see footnote \ref{NormFoot}).

\begin{definition}\label{defn:Gamma}
Let $A$ be a separable $\mathrm{C}^*$-algebra with $T(A)$ non-empty and compact.
We say that $A$ has \emph{uniform property $\Gamma$} if for all $n\in\mathbb N$, there exist projections $p_1,\dots,p_n\in A^\omega \cap A'$ summing to $1_{A^\omega}$, such that
\begin{equation}\label{gamma.def}
\tau(ap_i)=\tfrac{1}{n}\tau(a),\quad a\in A,\ \tau\in T_\omega(A),\ i=1,\dots,n.
\end{equation}
\end{definition}
Note that the projections $p_1,\dots,p_n$ in Definition \ref{defn:Gamma} are necessarily pairwise orthogonal.

While we insist in Definition \ref{defn:Gamma} that every element $a\in A$ can be `tracially divided' in an approximately central fashion (as opposed to just requiring division of the unit), this does parallel the II$_1$ factor setting.  If $\mathcal M$ is a II$_1$ factor and $p\in \mathcal M^\omega\cap M'$, then $\tau_{\mathcal M^\omega}(p\,\cdot)$ defines a tracial functional on $\mathcal M$, and hence by uniqueness of the trace on $\mathcal M$,
\begin{equation}\label{gamma.VNA}
	\tau_{\mathcal M^\omega}(pa)=\tau_{\mathcal M^\omega}(p)\tau_{\mathcal M}(a),\quad a\in \mathcal M.
\end{equation}
So in the II$_1$ factor setting, $p$ divides arbitrary elements in trace, in just the same way it divides the unit.  This is not true for $\mathrm{C}^*$-algebras whose trace space is non-Bauer; \cite[Example 3.3]{CETWW} gives an example. 

\smallskip
Kirchberg's $\epsilon$-test or other reindexing methods allow us to replace the central sequence algebra with the relative commutant of an arbitrary separable subset. The argument is standard and we omit the proof (cf.\ \cite[Proposition 4.4]{Wi12} for an example of such a reindexing argument). 

\begin{lemma}
\label{lem:GammaS}
Let $A$ be a separable $\mathrm{C}^*$-algebra with $T(A)$ non-empty and compact. Suppose $A$ has uniform property $\Gamma$.
Then for $n \in \mathbb N$ and a $\|\cdot\|_{2,T_\omega(A)}$-separable subset $S$ of $A^\omega$, 
there exist projections $p_1,\dots,p_n\in A^\omega \cap S'$ summing to $1_{A^\omega}$, such that
\begin{equation}
\tau(ap_i)=\tfrac{1}{n}\tau(a),\quad a\in S,\ \tau\in T_\omega(A),\ i=1,\dots,n.
\end{equation}
\end{lemma}

In this paper, $\mathcal Z$-stable $\mathrm{C}^*$-algebras are the most important instances of uniform property $\Gamma$; for later use we record a stronger statement.
\begin{proposition}
\label{prop:GammaZ}
Suppose $A$ is a separable, $\mathcal Z$-stable $\mathrm{C}^*$-algebra with $T(A)$ non-empty and compact.  Then for any $\|\cdot\|_{2,T_\omega(A)}$-separable subset $S$ of $A^\omega$ and any $n \in \mathbb N$, there is a unital embedding $M_n \to A^\omega \cap S'$.
In particular, $A$ has uniform property $\Gamma$.
\end{proposition}

\begin{proof}
By \cite[Proposition 4.4 (4)]{Kir06}, there is a unital embedding $\mathcal Z\rightarrow (A_\omega \cap A')/A^\perp$.\footnote{Here $A^\perp$ is the annihilator $\{x\in A_\omega:xa=ax=0,\ a\in A\}$.} Since $A^\perp \subseteq J_{A}$ (by Lemma \ref{lem:ZeroInAomega:easy}), this induces a unital $^*$-homomorphism $\mathcal Z \to A^\omega \cap A'$. 
By a standard use of Kirchberg's $\epsilon$-test (\cite[Lemma A.1]{Kir06}), there is a unital $^*$-homomorphism $\phi:\mathcal Z \to A^\omega \cap S'$.
As the unit ball of $A^\omega$ is $\|\cdot\|_{2,T_\omega(A)}$-complete (Lemma \ref{BallComplete}), so too is the unit ball of $A^\omega \cap S'$. Since $\phi$ is $(\|\cdot\|_{2,\tau_{\mathcal{Z}}}$-$\|\cdot\|_{2,T_\omega(A)})$-contractive, using Kaplansky's density theorem, it follows that $\phi$ extends by continuity to a map from the II$_1$ factor $\pi_{\tau_{\mathcal Z}}(\mathcal Z)''$.\footnote{This is isomorphic to the hyperfinite II$_1$ factor $\mathcal R$, but this is not needed.} Composing with a unital embedding $M_n \to \mathcal \pi_{\tau_\mathcal{Z}}(\mathcal Z)''$, we obtain a unital $^*$-homomorphism $\psi:M_n \to A^\omega \cap S'$, showing the first conclusion.

To get uniform property $\Gamma$, take equivalent, orthogonal projections $q_1,\dots,q_n \in M_n$ which sum to $1_{M_n}$, and define $p_i := \psi(q_i)$. As $M_n$ has a unique trace, for any $\tau \in T(A^\omega)$ and any $a \in A$, $\tau(\psi(\cdot)a)$ must be a scalar multiple of the trace on $M_n$, so that
\begin{equation} \tau(p_ia) = \tau(\psi(q_i)a) = \tau(a)\tau_{M_n}(q_i) = \tfrac{1}{n}\tau(a).\qedhere \end{equation}
\end{proof}

We now turn to one of the key applications of uniform property $\Gamma$.

\begin{lemma}[Tracial projectionisation]
\label{lem:Projectionization}
Let $A$ be a separable $\mathrm{C}^*$-algebra with $T(A)$ non-empty and compact, and with uniform property $\Gamma$.
Let $S,T \subseteq A^\omega$ be $\|\cdot\|_{2,T_\omega(A)}$-separable subsets.
Let $b \in (A^\omega \cap S')_+^1$.
Then there is a projection $p \in A^\omega \cap S'$ such that
\begin{equation}\label{eq:ProjectionizationTrace} \tau(ap) = \tau(ab) \quad a \in T,\ \tau \in T_\omega(A). \end{equation}
\end{lemma}

\begin{proof}
Fix $n\in\mathbb N$.  By Kirchberg's $\epsilon$-test (\cite[Lemma A.1]{Kir06}), it suffices to produce $p \in (A^\omega \cap S')_+^1$ satisfying \eqref{eq:ProjectionizationTrace} and $\|p-p^2\|_{2,T_\omega(A)}^2 \leq 1/n$.

Define functions $f_1,\dots,f_n \in C([0,1])$ by requiring that
$f_i|_{[0,\frac{i-1}n]} \equiv 0$, $f_i|_{[\frac in,1]} \equiv 1$, and $f_i$ is linear on $[\frac{i-1}n,\frac in]$. Note that
\begin{gather}
\label{eq:Projectionization2}
\mathrm{id}_{[0,1]} = \frac1n \ssum_{i=1}^n f_i \quad \text{and} \\
\label{eq:Projectionization3}
0\leq \ssum_{i=1}^n f_i-f_i^2 \leq 1_{C([0,1])}.
\end{gather}

Lemma \ref{lem:GammaS} provides pairwise orthogonal projections $p_1,\dots,p_n \in A^\omega \cap S' \cap \{b\}'$ satisfying
\begin{equation} \label{eq:Projectionization1} \tau(p_ix) = \tfrac1n \tau(x), \quad \tau \in T_\omega(A),\ x \in \mathrm{C}^*(T \cup \{b\}). \end{equation}
Define
$p:=\sum_{i=1}^n p_if_i(b) \in A^\omega \cap S'$.
Then for $a \in T$ and $\tau\in T_\omega(A)$,
\begin{eqnarray}
\tau(pa) = \ssum_{i=1}^n \tau(p_if_i(b)a) 
\stackrel{\eqref{eq:Projectionization1}}= \ssum_{i=1}^n \frac1n \tau(f_i(b)a)
\stackrel{\eqref{eq:Projectionization2}}= \tau(ba).
\end{eqnarray}
Moreover, for $\tau \in T_\omega(A)$, using properties of the $p_i$ in the first line,
\begin{eqnarray}
\notag
\tau(p-p^2)
&=& \tau\big(\ssum_{i=1}^n p_i(f_i(b)-f_i(b)^2)\big) \\
&\stackrel{\eqref{eq:Projectionization1}}=& \ssum_{i=1}^n \frac1n \tau(f_i(b)-f_i(b)^2) \stackrel{\eqref{eq:Projectionization2}}\leq \tfrac1n \tau(1_{A^\omega}) = \tfrac1n.
\end{eqnarray}
Since $p-p^2$ is a positive contraction, we get
\begin{equation} \|p-p^2\|_{2,T_\omega(A)}^2=\sup_{\tau\in T_\omega(A)}\tau((p-p^2)^2) \leq \sup_{\tau \in T_\omega(A)} \tau(p-p^2) \leq \tfrac1n, \end{equation} as required.
\end{proof}
\section{Complemented partitions of unity}\label{S3}

\noindent
In this section we perform the main technical argument of the paper, obtaining complemented partitions of unity for nuclear $\mathrm{C}^*$-algebras with uniform property $\Gamma$.  Recall our standard convention from \eqref{def:uniformtracialultrapower} that $A^\omega$ denotes the uniform tracial ultrapower, which, by Proposition \ref{lem:ZeroInAomega}, is unital when $T(A)$ is non-empty and compact, and $A^\omega\cap A'$ consists of the $A$-central elements of $A^\omega$ (see footnote \ref{NormFoot}).

\begin{definition}\label{defn:CPOU}
	Let $A$ be a separable $\mathrm{C}^*$-algebra with $T(A)$ non-empty and compact.
	We say that $A$ has \emph{complemented partitions of unity} (CPoU) if for every $\|\cdot\|_{2,T_\omega(A)}$-separable subset $S$ of $A^\omega$, every family $a_1,\dots,a_k \in (A^\omega)_+$, and any scalar
	\begin{equation}
	\label{eq:CPoUTraceIneq1}
	\delta>\sup_{\tau\in T_\omega(A)}\min\{\tau(a_1),\dots,\tau(a_k)\},
	\end{equation}
	there exist orthogonal projections $p_1,\dots,p_k\in A^\omega\cap S'$ such that
	\begin{equation}
	p_1+\cdots+p_k = 1_{A^\omega} \ \text{and}\ 
	\tau(a_ip_i)\leq \delta\tau(p_i),\ \tau\in T_\omega(A),\ i=1,\dots,k.
	\end{equation}
\end{definition}

Standard methods yield the following characterisation of CPoU with approximations in place of the use of the uniform tracial ultrapower; this characterisation matches the informal version of CPoU given in Definition \ref{intro:defCPoU}.
Note that central surjectivity (Lemma \ref{lem:CentralSurjection}) shows that condition \eqref{prop:CPoUReformulation.E2} may be interchanged with $\|[e_i,x]\|_{2,T(A)} < \epsilon$.

\begin{proposition}
\label{prop:CPoUReformulation}
Let $A$ be a separable $\mathrm{C}^*$-algebra with $T(A)$ non-empty and compact.
Then $A$ has CPoU if and only if for every finite set $\mathcal F \subset A$, every $\epsilon>0$, every family $a_1,\dots,a_k\in A_+$, and any scalar
\begin{equation}\label{prop:CPoUReformulation.E1} \delta > \sup_{\tau \in T(A)} \min\{\tau(a_1),\dots,\tau(a_k)\}, \end{equation}
there exist pairwise orthogonal $e_1,\dots,e_k \in A_+^1$  such that
\begin{align}
\label{prop:CPoUReformulation.E2}
	\|[e_i,x]\| &< \epsilon, \quad &&x \in \mathcal F,\ i=1,\dots,k, \\
\notag
	\tau(e_1+\cdots+e_k) &> 1-\epsilon, \quad && \tau \in T(A), \text{ and} \\
	\tau(a_ie_i) &< \delta\tau(e_i)+\epsilon,\quad && \tau\in T(A),\ i=1,\dots,k.
\end{align}
\end{proposition}

\begin{remark}\label{prop:CPoUReformulation:rem} In particular, it suffices to verify CPoU with the positive elements $a_1,\dots,a_k$ taken in $A$ rather than in $A^\omega$ (and in this situation \eqref{prop:CPoUReformulation.E1} is equivalent to (\ref{eq:CPoUTraceIneq1})).
\end{remark}

For later use, we note the following stability result regarding CPoU.

\begin{lemma}\label{CPOU:Matrix}
Let $A$ be a separable $\mathrm{C}^*$-algebra with $T(A)$ non-empty and compact. For $m\in\mathbb N$, $A$ has CPoU if and only if $M_m(A)$ has CPoU.
\end{lemma}
\begin{proof}
Fix $m\in\mathbb N$ and note that $M_m(A)^\omega\cong M_m(A^\omega)\cong A^\omega\otimes M_m$ with a corresponding identification between the amplifications of limit traces in $T_\omega(A)$ to $M_m(A^\omega)$ and the limit traces $T_\omega(M_m(A))$.  

Suppose that $A$ has CPoU. 
Let $a^{(1)},\dots,a^{(k)}\in (A^\omega\otimes M_m)_+$ and $\delta>0$ be such that for all $\tau\in T_\omega(M_m(A))$, there exists $i\in\{1,\dots,k\}$ such that $\tau(a^{(i)})<\delta$.  
Write $a^{(i)}=\sum_{r,s=1}^ma_{r,s}^{(i)}\otimes e_{r,s}$ for $a_{r,s}^{(i)}\in A^\omega$ so that for every $\tau\in T_\omega(A)$, there exists $i\in\{1,\dots,k\}$ with $\tau(\sum_{r=1}^ma^{(i)}_{r,r})<m\delta$.  Given any separable subset $S$ of $M_m(A^\omega)$, let $\tilde{S}\subset A^\omega$ consist of the matrix entries of elements 
of $S$. By CPoU for $A$, there exists a partition of unity consisting of projections $p_1,\dots,p_k\in A^\omega\cap \tilde{S}'$ such that \begin{equation}
\tau\big(\ssum_{r=1}^ma_{r,r}^{(i)}p_i\big)\leq m\delta\tau(p_i),\quad \tau\in T_\omega(A),\ i=1,\dots,k.
\end{equation}
Then $p_1\otimes 1_{M_m},\dots,p_k\otimes 1_{M_m}$ are projections in $(A^\omega\otimes M_m)\cap S'$ with \begin{equation}
\tau(a^{(i)}(p_i\otimes 1_{M_m}))\leq\delta\tau(p_i\otimes 1_{M_m}),\quad \tau\in T_\omega(A\otimes M_m),
\end{equation}
for each $i$.  Therefore $M_m(A^\omega)\cong M_m(A)^\omega$ has CPoU.

Conversely suppose $M_m(A)^\omega\cong A^\omega\otimes M_m$ has CPoU. 
Let $a_1,\dots,a_k\in A^\omega$ be positive elements and $\delta>0$ such that \eqref{eq:CPoUTraceIneq1} holds. Then
\begin{equation}
\sup_{\tau\in T_\omega(M_m(A))}\min\{\tau(a_1\otimes 1_{M_m}),\dots,\tau(a_k\otimes 1_{M_m})\}<\delta.
\end{equation}
Given a separable subset $S\subset A^\omega$ containing $1_{A^\omega}$ we can apply CPoU for $M_m(A)$ to find a partition of unity consisting of projections $\tilde{p}_1,\dots,\tilde{p}_k$ in $(A^\omega\otimes M_m)\cap (S\otimes M_m)'$ with
\begin{equation}
\tau((a_i\otimes 1_{M_m})\tilde{p}_i)\leq \delta\tau(\tilde{p}_i),\quad\tau\in T_\omega(M_m(A)),\ i=1,\dots,k.
\end{equation}
Since each $\tilde{p}_i$ commutes with $1_{A^\omega}\otimes M_m$, it is of the form $p_i\otimes 1_{M_m}$ for some $p_i\in A^\omega\cap S'$.  The family $p_1,\dots,p_k$ witnesses CPoU for $A$.
\end{proof}

We now turn towards the proof of Theorem \ref{thm:mainCPoU}, beginning by recording the strong form of the completely positive approximation property we require. In the case when $\hat{S}=A$, and working with approximate statements rather than ultrapowers, the following is a consequence of \cite[Theorem 3.1]{BCW16} (which is stated in terms of finite sets and $\epsilon$'s, and obtains a stronger convexity conclusion in (ii), which we do not need).   The version stated here is obtained from \cite[Theorem 3.1]{BCW16} via standard reindexing or applying Kirchberg's $\epsilon$-test.  We omit the details.

\begin{lemma}[{\cite[Theorem 3.1]{BCW16}}]
	\label{lem:BCW}
	Let $A$ be a separable, nuclear $\mathrm{C}^*$-algebra, let $\hat{S}$ be a $\|\cdot\|$-separable subalgebra of $A_\omega$, and let $S \subset A^\omega$ be the $\|\cdot\|_{2,T_\omega(A)}$-closure of the image of $\hat{S}$ in $A^\omega$.
	Then there exists a sequence $(F_m, \psi_m, \phi_m )_{m=1}^\infty$, where each $F_m$ is a finite dimensional $\mathrm{C}^*$-algebra, and $\psi_m:A\rightarrow F_m$ and $\phi_m:F_m\rightarrow A$ are c.p.c.\ maps which induce c.p.c.\ maps $\psi:A_\omega\rightarrow\prod_\omega F_m$ and $\phi:\prod_\omega F_m\rightarrow A^\omega$ such that\footnote{Here $\prod_\omega F_m$ is the norm ultraproduct of the sequence $(F_m)_{m=1}^\infty$, defined analogously to the ultrapower.}
	\begin{enumerate}
		\item[\rm{(i)}] the restriction of $\phi\circ\psi$ to $\hat{S}$ agrees with the (restricted) quotient map $\hat{S}\to S \subset A^\omega$,
		\item[\rm{(ii)}] the $\phi_m$ are sums of c.p.c.\ order zero maps,
		\item[\rm{(iii)}] the map $\psi$ restricts to an order zero map on $\hat{S}$.
	\end{enumerate}
\end{lemma}

In our use of these approximations  to perform the first step in the proof of Theorem \ref{thm:mainCPoU} (as outlined in the introduction), the key role of (ii) and (iii) is that the order zero maps preserves traciality (\cite[Corollary 4.4]{WZ09}). Consequently, $\tau\circ\phi_m$ is a tracial functional on $F_m$ if $\tau \in T(A)$, and $\tau \circ \psi$ is a tracial functional on $\hat{S}$ if $\tau \in T(\prod_\omega F_m$).  

\begin{lemma}\label{s-lemma}
	Let $A$ be a separable, nuclear $\mathrm{C}^*$-algebra with $T(A)$ non-empty and compact.
	Let $S \subset A^\omega$ be a $\|\cdot\|_{2,T_\omega(A)}$-separable subset, and $q \in A^\omega \cap A'$ a projection such that there exists $\mu\in (0,1]$ with $\tau(q)=\mu$ for all $\tau\in T_\omega(A)$. Let $a_1,\dots,a_k \in A_+$ and $\delta>0$ be such that
	\begin{align}
	\delta>\sup_{\tau\in T(A)}\min\{\tau(a_1),\dots,\tau(a_k)\},\quad \tau \in T_\omega(A).
	\end{align}
	Then there exist $b_1, \dots, b_k \in (A^\omega \cap S')_+^1$ such that
	\begin{align}
	\tau\big(\ssum_{i=1}^{k} b_iq\big) &= \mu, \qquad &&\tau\in T_\omega(A), \text{ and}\label{eq:s-lemma1} \\
	\tau(a_i b_iq) &\leq \delta \tau(b_iq), \qquad &&\tau \in T_\omega (A),\ i=1,\dots,k.\label{eq:s-lemma2}
	\end{align}
\end{lemma}

\begin{proof}
	Without loss of generality, we may assume that $A \cup \{q\} \subseteq S$. By central surjectivity (Lemma \ref{lem:CentralSurjection}), choose $\hat{q}\in (A_\omega\cap A')_+^1$ which quotients to $q\in A^\omega$.  
	Let $\hat{S} \subset A_\omega$ be a $\|\cdot\|$-separable $\mathrm{C}^*$-algebra containing $A\cup\{\hat{q}\}$ and whose image in $A^\omega$ is $\|\cdot\|_{2,T_\omega(A)}$-dense in $S$.
	
	Fix $0< \epsilon < \mu/2$. By Kirchberg's $\epsilon$-test, it suffices to find $b_1,\dots,b_k \in (A^\omega\cap S')_+^1$  satisfying 
	the weaker conditions 
	\begin{align}
	\label{eq:s-lemma1weak}
	\tau\big(\ssum_{i=1}^kb_iq\big)&\in [\mu-2\epsilon,\mu] \quad &&\tau\in T_\omega(A), \quad \text{and} \\
	\tau(a_i b_i q) &\leq \delta(1+\epsilon)\tau(b_iq), \qquad &&\tau \in T_\omega (A),\ i=1,\dots,k.\label{eq:s-lemma2weak}
	\end{align}
	
	Lift $\hat{q}$ to a sequence $(q_m)_{m=1}^\infty$ in $A_+^1$.
 	Let $(F_m,\psi_m,\phi_m)_{m=1}^\infty$ be as in Lemma \ref{lem:BCW}.
	For each $m$, decompose $F_m$ as $\bigoplus_{\lambda\in\Lambda_m}F_{m,\lambda}$, where $\Lambda_m$ is a finite set and each $F_{m,\lambda}$ is a full matrix algebra.
	Correspondingly, $\psi_m$ decomposes as a direct sum $\bigoplus_{\lambda\in\Lambda_m} \psi_{m,\lambda}$.
	We denote the tracial state on $F_{m,\lambda}$ by $\tau_{F_{m,\lambda}}$.
	As $T(A)$ is compact, applying Dini's theorem to an approximate unit for $A$ gives $e \in A_+^1$ such that
	\begin{equation}
	\label{eq:s-lemma3} \tau(e) \geq \tfrac1{1+\epsilon}, \quad \tau \in T(A). \end{equation}
Then
	\begin{align}
	\tau(eq) =\tau(q)-\tau((1_{A^\omega}-e)q)&\geq \mu-\tau(1_{A^\omega}-e)\nonumber\\&\geq\mu -\tfrac{\epsilon}{1+\epsilon}\geq\mu-\epsilon,\quad \tau\in T_\omega(A). \label{eq:s-lemma4} 
	 \end{align}

	Let $\tau=(\tau_m)_{m=1}^\infty\in T_\omega(A)$ be a limit trace. For each $m$, by Lemma \ref{lem:BCW}(ii) and \cite[Corollary 4.4]{WZ09}, $\tau_m\circ\phi_m$ is a tracial functional on $F_m$ of norm at most $1$ (since $\phi_m$ is contractive), and hence of the form 
	\begin{equation}\label{s-lemma:saw.8}
	\tau_m \circ \phi_m = \ssum_{\lambda \in \Lambda_m} \alpha^{(m,\tau_m)}_\lambda \tau_{F_{m,\lambda}},
	\end{equation}
	for positive real numbers $\alpha_\lambda^{(m,\tau_m)}$ satisfying
	\begin{equation}\label{s-lemma:saw.10}
	\ssum_{\lambda\in\Lambda_m}\alpha_\lambda^{(m,\tau_m)}=\|\tau_m\circ\phi_m\| \leq 1.
	\end{equation} 
	
	For each $m$, define
	\begin{equation}
	\tilde{\Lambda}_m := \{\lambda\in\Lambda_m : \tau_{F_{m,\lambda}} ( \psi_{m,\lambda}(eq_m)) \geq \epsilon  \}. \label{s-lemma:saw.2}
	\end{equation}
We first claim that 
\begin{equation}\label{eq:s-lem-newsaw.1}
\{m\in\mathbb N:\tilde{\Lambda}_m\neq\emptyset\}\in\omega.
\end{equation}  If this is not the case, then for any $\tau=(\tau_m)_{m=1}^\infty\in T_\omega(A)$, as above, using that $\phi\psi(e\hat{q})=eq$, we have
\begin{align}
\tau(eq)&=(\tau\circ\phi)(\psi(e\hat{q}))=\lim_{m\rightarrow \omega}\ssum_{\lambda\in\Lambda_m}\alpha_\lambda^{(m,\tau_m)}\tau_{F_{m,\lambda}}(\psi_m(eq_m))\leq\epsilon,
\end{align}
which combines with (\ref{eq:s-lemma4}) to contradict the choice of $\epsilon<\mu/2$.

Inductively define
	for $i=1,\dots,k$,
	\begin{align} \label{s-lemma:saw.7} 
	&\tilde{\Lambda}_m^{(i)} := \\
\notag &\hspace*{-4em} \big\{ \lambda \in \tilde{\Lambda}_m : \tau_{F_{m,\lambda}} (\psi_{m,\lambda} (a_iq_m)) \leq \delta(1+\epsilon)\tau_{F_{m,\lambda}} (\psi_{m,\lambda} (q_m))  \big\} 
	\setminus \textstyle\bigcup\limits_{j=1}^{i-1} \tilde{\Lambda}_m^{(j)}. \notag
	\end{align}
We next claim that 
	\begin{equation}\label{s-lemma:claim}
	\big\{m:\tilde{\Lambda}_m=\textstyle\bigcup\limits_{i=1}^k\tilde{\Lambda}_m^{(i)}\big\}\in\omega.
	\end{equation} If this is not the case, then there exists
	$\lambda_m \in \Lambda_m$ for each $m$, such that 
	\begin{equation}
	\label{eq:s-lemma-contradiction}
	\big\{m: \lambda_m \in \tilde{\Lambda}_m \setminus \textstyle\bigcup\limits_{i=1}^k \tilde{\Lambda}_m^{(i)}\big\}\in\omega
	\end{equation}
(this uses \eqref{eq:s-lem-newsaw.1}).
	Define a linear functional $\sigma$ on $A$ by
	\begin{equation}
	\sigma(b) := \lim_{m \rightarrow \omega} \tau_{F_{m,\lambda_m}} (\psi_{m,\lambda_m} (bq_m)),\quad b\in A.
	\end{equation}
	By Lemma \ref{lem:BCW}(iii) and since $\hat{q}\in A_\omega\cap A'$, $\sigma$ factors as an order zero map $A\rightarrow \prod_{\omega}F_{m,\lambda_m}$ followed by the canonical limit trace on this ultraproduct. Thus it is a tracial functional by \cite[Corollary 4.4]{WZ09}.  Now 
	\begin{equation}
	\sigma(e)=\lim_{m\rightarrow\omega}\tau_{F_{m,\lambda_m}}(\psi_{m,\lambda_m}(eq_m))\geq\epsilon,
	\end{equation}
	by \eqref{s-lemma:saw.2}, as $\{m:\lambda_m\in\tilde{\Lambda}_m\}\in\omega$. In particular, $\sigma\neq 0$. By hypothesis, there exists $i\in\{1,\dots,k\}$ such that $\sigma(a_i) < \delta \|\sigma\|$.
	Since
	\begin{equation} \|\sigma\| \stackrel{\eqref{eq:s-lemma3}}\leq (1+\epsilon)\sigma(e) = (1+\epsilon) \lim_{m\to\omega} \tau_{F_{m,\lambda_m}}(\psi_{m,\lambda_m}(eq_m)),  \end{equation} we get from $\sigma(a_i)<\delta\|\sigma\|$ that
	\begin{equation}
	\lim_{m \rightarrow \omega} \tau_{F_{m,\lambda_m}} (\psi_{m,\lambda_m} (a_iq_m)) < \delta(1+\epsilon)\lim_{m \rightarrow \omega} \tau_{F_{m,\lambda_m}} (\psi_{m,\lambda_m} (q_m)).
	\end{equation}
	But then the definition of the sets $\tilde{\Lambda}_m^{(j)}$ ensures that $\lambda_m \in \bigcup_{j=1}^i \tilde{\Lambda}_m^{(j)}$ for $\omega$-many $m$.  This contradicts \eqref{eq:s-lemma-contradiction}, and proves the claim.

	Now, for $i=1,\dots,k$, and $m \in \mathbb N$, define 
	\begin{align}
	f_m^{(i)} &:= \textstyle\bigoplus\limits_{\lambda \in \tilde{\Lambda}_m^{(i)}} 1_{F_{m,\lambda}}\in (F_m)_+^1\quad \text{and} \label{s-lemma:saw-5}\quad b_m^{(i)} := \phi_m(f_m^{(i)}) \in A_+^1.
	\end{align}
	Set $f_i:=(f_m^{(i)})_{m=1}^\infty \in \prod_\omega F_m$ and $b_i:=(b_m^{(i)})_{m=1}^\infty\in A^\omega$, so $b_i=\phi(f_i)$.
Since each $\phi_m$ is a contractive map, note that $\sum_{i=1}^k b_i$ has norm at most $1$, and thus $\tau(\sum_{i=1}^k b_iq) \leq \tau(q) = \mu$ for $\tau\in T_\omega(A)$.
	
	By Lemma \ref{lem:BCW}(i), it follows by the Stinespring argument of \cite[Lemma 3.5]{KW04}, that $\psi$ maps $\hat{S}$ into the multiplicative domain of $\phi$,
	and so for $x \in \hat{S}$ and $\bar{x} \in S$ its image,
	\begin{equation}\label{s-lemma:saw.3}
	\phi(\psi(x)y)=\phi(\psi(x))\phi(y)=\bar{x}\phi(y),\quad y\in\textstyle\prod_{\omega}F_m.
	\end{equation}
	In particular, as $f_i$ is central in $\prod_\omega F_m$ we have
	\begin{equation}
	\bar xb_i=\bar x\phi(f_i)\stackrel{\eqref{s-lemma:saw.3}}=\phi(\psi(x)f_i)=\phi(f_i\psi(x))\stackrel{\eqref{s-lemma:saw.3}}=\phi(f_i)\bar x=b_i\bar x.
	\end{equation}
	Since the image of $\hat{S}$ is $\|\cdot\|_{2,T_\omega(A)}$-dense in $S$, it follows that $b_i\in A^\omega\cap S'$. 

	For $i=1,\dots,k$, and $\tau=(\tau_m)_{m=1}^\infty\in T_\omega(A)$, we have
	\begin{eqnarray}
	\notag
	\tau(a_i b_i q) \hspace*{-1em} &=&\tau(a_i q\phi(f_i)) \stackrel{\eqref{s-lemma:saw.3}}= \tau(\phi(\psi(a_i\hat{q})f_i)) \\
	\notag
	&=&\lim_{m\rightarrow\omega}\tau_m\big(\phi_m(\psi_m(a_iq_m)f_m^{(i)})\big)\\
	\notag
	&\stackrel{(\ref{s-lemma:saw.8}),(\ref{s-lemma:saw-5})}{=}&\lim_{m\rightarrow\omega}\ssum_{\lambda\in \tilde{\Lambda}_m^{(i)}}\alpha_\lambda^{(m,\tau_m)}\tau_{F_{m,\lambda}}(\psi_{m,\lambda}(a_iq_m))\\
	\notag
	&\stackrel{(\ref{s-lemma:saw.7})}{\leq}&\delta(1+\epsilon)\lim_{m\rightarrow\omega}\ssum_{\lambda\in \tilde{\Lambda}_m^{(i)}}\alpha_\lambda^{(m,\tau_m)}\tau_{F_{m,\lambda}}(\psi_{m,\lambda}(q_m))\\
	&\stackrel{(\ref{s-lemma:saw.8}),(\ref{s-lemma:saw-5})}{=}&\delta(1+\epsilon)\lim_{m\rightarrow\omega} \tau_m\big(\phi_m(\psi_m(q_m)f_m^{(i)})\big)
	\stackrel{(\ref{s-lemma:saw.3})}{=} \delta(1+\epsilon)\tau(b_i q),
	\end{eqnarray}
	proving (\ref{eq:s-lemma2weak}).	Likewise, we have
	\begin{eqnarray}
	\tau\big(\ssum_{i=1}^k b_i q\big)\nonumber
	&\geq& \ssum_{i=1}^k \tau(e b_i q) \nonumber\\
	&\stackrel{\eqref{s-lemma:saw.3}}=& \lim_{m \rightarrow \omega} \ssum_{i=1}^k (\tau_m \circ \phi_m)\big(\psi_m(eq_m)f_m^{(i)}\big) \nonumber\\
	&\stackrel{(\ref{s-lemma:saw.8}),(\ref{s-lemma:saw-5})}{=}& \lim_{m \rightarrow \omega} \ssum_{i=1}^k \ssum_{\lambda\in \tilde{\Lambda}_m^{(i)}} \alpha^{(m,\tau_m)}_\lambda \tau_{F_{m,\lambda}} (\psi_{m,\lambda} (eq_m))\nonumber\\
	&\stackrel{(\ref{s-lemma:claim})}{=}& \lim_{m \rightarrow \omega} \ \ssum_{\lambda\in \Lambda_m} \alpha^{(m,\tau_m)}_\lambda \tau_{F_{m,\lambda}} (\psi_{m,\lambda} (eq_m))\nonumber\\
	&& \quad-\lim_{m \rightarrow \omega} \ssum_{\lambda\in \Lambda_m\setminus\tilde{\Lambda}_m} \alpha^{(m,\tau_m)}_\lambda \tau_{F_{m,\lambda}} (\psi_{m,\lambda} (eq_m))\nonumber\\
	&\stackrel{(\ref{s-lemma:saw.8}),(\ref{s-lemma:saw.2})}{\geq}&\tau(\phi(\psi(e\hat{q})))-\epsilon\lim_{m\rightarrow\omega}\ssum_{\lambda\in\Lambda_m\setminus \tilde{\Lambda}_m}\alpha^{(m,\tau_m)}_\lambda\nonumber\\
	&\stackrel{\eqref{eq:s-lemma4},(\ref{s-lemma:saw.10})}\geq& \mu-2\epsilon,
	\end{eqnarray} establishing (\ref{eq:s-lemma1weak}).
\end{proof}

We now show how the conclusion of the previous lemma combines with uniform property $\Gamma$ to give CPoU.  The strategy has the spirit of geometric series arguments used in \cite{Wi12,Wi16}, but as we work only with projections of constant value on (limit) traces we can use a maximality argument directly.  Note that we do not require $A$ to be nuclear for this step, only that the conclusion of Lemma \ref{s-lemma} has been obtained.

\begin{lemma}\label{lem:GeomSeriesCPOU}
	Let $A$ be a separable $\mathrm{C}^*$-algebra with $T(A)$ non-empty and compact.
	Suppose that $A$ has uniform property $\Gamma$ and that the conclusion of Lemma \ref{s-lemma} holds.
	Then $A$ has CPoU.
\end{lemma}

\begin{proof}
As in Remark \ref{prop:CPoUReformulation:rem}, it suffices to verify CPoU for positive elements $a_1, \dots, a_k \in A$ and $\delta>0$ such that
	\begin{equation}
	\delta>\sup_{\tau\in T(A)}\min\{\tau(a_1),\dots,\tau(a_k)\}.
	\end{equation}
	Let $S\subset A^\omega$ be a $\|\cdot\|_{2,T_\omega(A)}$-separable subset.

	Let $I$ denote the set of all $\alpha\in [0,1]$ such that there exist pairwise orthogonal projections $p_1, \dots, p_k \in A^\omega \cap S'$ such that for $i=1,\dots,k$,
	\begin{equation}\label{POU.defI}
	\tau\big(\ssum_{i=1}^k p_i\big) = \alpha \quad \text{and} \quad \tau(a_i p_i) \leq \delta \tau(p_i) ,\quad \tau\in T_\omega(A).
	\end{equation}
	First of all, notice that $0\in I$, as $p_1 = p_2 = \dots = p_k :=0$ satisfy \eqref{POU.defI}, so $I\neq\emptyset$. By Kirchberg's $\epsilon$-test, $I$ is closed, so that $\beta:=\sup I\in I$.  	Our objective is to show that $\beta=1$, as then this gives CPoU.

Suppose for a contradiction that $\beta<1$ and let $p_1^{(1)}, \dots, p_k^{(1)} \in A^\omega \cap S'$ be orthogonal projections such that 
	\begin{equation}
	\tau\big(\ssum_{i=1}^k p^{(1)}_i\big) = \beta \quad \text{and} \quad \tau(a_i p^{(1)}_i) \leq \delta \tau ( p^{(1)}_i )\label{1/k-lemma eq1}
	\end{equation}
	for all $\tau \in T_\omega (A)$.
	Using the conclusion of Lemma \ref{s-lemma} (which we are assuming holds for $A$), with $q:=1_{A^\omega}-(p^{(1)}_1+\cdots+p^{(1)}_k)$ and $S\cup \{q\}$ in place of $S$, there exist $b^{(2)}_1, \cdots, b^{(2)}_k \in (A^\omega \cap S' \cap \{q\}')_+^1$
  such that
	\begin{align}
	\tau\big(\ssum_{i=1}^k b^{(2)}_i q\big) = 1-\beta \quad \text{and} \quad \tau(a_i b^{(2)}_i q) \leq \delta \tau(b^{(2)}_iq) \label{1/k-lemma eq2}
	\end{align}
	for all $\tau \in T_\omega (A)$ and $i=1, \dots, k$.
	As $A$ has uniform property $\Gamma$, for each $i$, we can apply Lemma \ref{lem:Projectionization} to $b^{(2)}_i$ to obtain 
	a projection $p^{(2)}_i \in A^\omega \cap S' \cap \{q\}'$ such that 
	\begin{equation}
	\tau(qp^{(2)}_i)=\tau(qb^{(2)}_i)\text{ and }\tau(a_iqp^{(2)}_i)=\tau(a_iqb^{(2)}_i),\quad \tau\in T_\omega(A). 	\label{1/k-lemma eq3}
	\end{equation} Set
	\begin{equation}
	\quad T:= A\cup S\cup\{q,p^{(1)}_1, \dots, p^{(1)}_k, p^{(2)}_1, \dots, p^{(2)}_k\} \subset A^\omega. \end{equation}
	Using uniform property $\Gamma$, Lemma \ref{lem:GammaS} gives pairwise orthogonal projections $r_1,\dots,r_k \in A^\omega\cap T'$ such that
	\begin{equation}\label{1/k-lem.saw.2}
	\tau(r_iy)=\tfrac{1}{k}\tau(y),\quad \tau\in T_\omega(A),\ y\in T,\ i=1,\dots,k.
	\end{equation}
	Define
	\begin{align}
	p_i & := p^{(1)}_i + q p^{(2)}_i r_i, \qquad i =1, \dots, k. \label{1/k-lemma eq4}
	\end{align}
	By construction, $p_1,\dots,p_k \in A^\omega \cap S'$ are pairwise orthogonal projections.
	We will verify that they satisfy (\ref{POU.defI}) for the value $\alpha:=\beta+\frac{1}{k}(1-\beta)>\beta$, providing the required contradiction.

	Let $\tau \in T_\omega (A)$. Then
	\begin{eqnarray}
	\tau\big(\ssum_{i=1}^k p_i\big) & \stackrel{\eqref{1/k-lemma eq4},(\ref{1/k-lem.saw.2})}{=}& \tau\big(\ssum_{i=1}^n p^{(1)}_i\big)+\frac{1}{k}\tau\big(\ssum_{i=1}^k q p^{(2)}_i\big) \notag \\
	& \stackrel{\eqref{1/k-lemma eq1},\eqref{1/k-lemma eq2},\eqref{1/k-lemma eq3}}{=} &\beta + \tfrac{1-\beta}{k},
	\end{eqnarray}
	establishing the first part of \eqref{POU.defI}.
	Also, for $i=1,\dots,k$,
	\begin{eqnarray}
	\tau(a_i p_i) &  \stackrel{\eqref{1/k-lemma eq4}}{=}& \tau(a_i p^{(1)}_i) + \tau(a_i q p^{(2)}_i r_{i}) \notag \\
	&\stackrel{\eqref{1/k-lemma eq3},\eqref{1/k-lem.saw.2}}=& \tau(a_i p^{(1)}_i) + \tfrac1k \tau(a_i q b^{(2)}_i) \notag \\
	& \stackrel{\eqref{1/k-lemma eq1},\eqref{1/k-lemma eq2}}{\leq}& \delta \tau(p^{(1)}_i) + \tfrac{1}{k}\delta \tau(b^{(2)}_iq ) \notag \\
	& \stackrel{\eqref{1/k-lemma eq3},\eqref{1/k-lem.saw.2}}{=}&\delta \tau(p^{(1)}_i + p^{(2)}_iq r_i) 
	\stackrel{\eqref{1/k-lemma eq4}}{=} \delta \tau(p_i). 
	\end{eqnarray}
	Thus $p_1,\dots,p_k$ witnesses that $\beta+\frac{1}{k}(1-\beta)\in I$, a contradiction.
\end{proof}

Combining the two previous results yields our main technical result, which we will subsequently use to obtain nuclear dimension estimates.
Theorem \ref{thm:mainCPoU} follows immediately from Theorem \ref{thm:GammaImpliesCPoUTechnical} and Proposition \ref{prop:GammaZ}.

\begin{theorem}
	\label{thm:GammaImpliesCPoUTechnical}
	Let $A$ be a separable, nuclear $\mathrm{C}^*$-algebra with $T(A)$ non-empty and compact, and with uniform property $\Gamma$.
	Then $A$ has CPoU. 
	\end{theorem}
	\begin{proof}
	This follows by combining Lemma \ref{s-lemma} with Lemma \ref{lem:GeomSeriesCPOU}.
\end{proof}	

\section{Structural results for relative commutants and classification of maps out of cones}\label{S4}

\noindent
With the key ingredient of complemented partitions of unity in place, we now turn to showing how to use it to obtain our main results. In this section we obtain the key structural results for $B^\omega$ (and its relative commutants), and transfer these to $B_\omega$ (and its relative commutants) leading to classification results for order zero maps, which will be used to obtain finite nuclear dimension in Section \ref{S5}.  

The strategy closely follows the arguments of \cite{BBSTWW}. The key difference is our use of CPoU to remove the hypothesis that the traces form a Bauer simplex from \cite{BBSTWW}. In comparing the results of this section with \cite{BBSTWW}, recall from Section \ref{S1.2}\eqref{TracialUltrapowerEx2} that when $T(B)$ is a Bauer simplex, $B^\omega$ is canonically identified with the ultraproduct $\mathcal M^\omega$ of the $\mathrm{W}^*$-bundle $\mathcal M$ obtained from $B$, and the results of \cite[Section 3]{BBSTWW} are all expressed in terms of this $\mathrm{W}^*$-bundle.\footnote{The hypothesis that $B$ has CPoU used here, replaces the hypothesis that $\mathcal M$ is McDuff used in \cite{BBSTWW}; both are obtained from $\mathcal Z$-stability of $B$.}

To make it transparent how we adapt the proofs of Sections 4 and 5 of \cite{BBSTWW}, we start by showing how CPoU can be used to obtain a version of the key local-to-global transfer lemma (\cite[Lemma 3.18]{BBSTWW}), used (there) to glue fibrewise properties of trivial $\mathrm{W}^*$-bundles to global properties.
In the following lemma, think of the equation $h_n(a,y)=0$ as a condition we want $y$ to satisfy, with $a$ a fixed constant. For example, $h(a,y)=ay-ya$ describes the condition $y\in \{a\}'$. The idea is that if for each trace there is an approximate solution to a family of conditions, then there are exact solutions in $B^\omega$ (i.e., $\|\cdot\|_{2,T(B)}$-approximate solutions in $B$).

\begin{lemma}[{cf.\ \cite[Lemma 3.18]{BBSTWW}}]
\label{lem:RealizingTypes}
Let $B$ be a separable, unital $\mathrm{C}^*$-algebra with $T(B)$ non-empty, and with CPoU. For each $m\in\mathbb N$, 
let \begin{equation}
h_m(x_1,\dots,x_{r_m},z_1,\dots,z_{s_m})
\end{equation} be a $^*$-polynomial in $r_m+s_m$ non-commuting variables.
Let $(a_i)_{i=1}^\infty$ be a sequence from $B^\omega$.
Suppose that, for every $\epsilon > 0$, $k \in \mathbb N$, and $\tau\in T(B^\omega)$ in the closure of $T_\omega(B)$, there exist contractions $y^\tau_1,y^\tau_2,\dots\in \pi_\tau(B^\omega)''$ (where $\pi_\tau$ is the GNS representation associated to $\tau$) such that
\begin{equation}
\label{eq:RealizingTypesFibre}
\|h_m(a_1,\dots,a_{r_m},y^\tau_1,\dots,y^\tau_{s_m})\|_{2,\tau} < \epsilon,\quad m=1,\dots,k.
\end{equation}
Then there exist contractions $y_i \in B^\omega$ for $i\in\mathbb N$ such that
\begin{equation}
\label{eq:RealizingTypesConclusion}
h_m(a_1,\dots,a_{r_m},y_1,\dots,y_{s_m}) =0,\quad m\in\mathbb N.
\end{equation}
\end{lemma}

\begin{remark}
\label{rmk:UsingRealizingTypes}
Before giving the proof, we point out the differences between Lemma \ref{lem:RealizingTypes} and \cite[Lemma 3.18]{BBSTWW} beyond the change in framework from $\mathrm{W}^*$-bundles to uniform tracial ultrapowers $B^\omega$.

(i) \cite[Lemma 3.18]{BBSTWW} was set up with exactly $2m$ variables in the $m^\text{th}$ polynomial. This was just a notational device, which we abandon here.

(ii) In \cite[Lemma 3.18]{BBSTWW}, the hypothesis and conclusion concern the absolute value of the trace applied to the $^*$-polynomial, whereas here, we quantify approximate local solutions using the 2-norm coming from each trace. This is a genuine difference forced on us by the setup of CPoU, and the absolute value version in \cite[Lemma 3.18]{BBSTWW} is formally stronger. 
However, in all applications of \cite[Lemma 3.18]{BBSTWW} in \cite{BBSTWW}, the $^*$-polynomials are of the form $k_m(\cdot)^*k_m(\cdot)$, where $k_m$ is itself a non-commutative $^*$-polynomial, and thus when we evaluate at the trace, we are taking the 2-norm squared of $k_m$. 

(iii) The most fundamental difference between Lemma \ref{lem:RealizingTypes} and \cite[Lemma 3.18]{BBSTWW} is the collection of traces we must use.  When $T(B)$ is compact, and we take $\mathcal M_n$ to be the strict closure of $B$ in \cite[Lemma 3.18]{BBSTWW}, then the hypothesis and conclusion of \cite[Lemma 3.18]{BBSTWW} only refer to those traces on $B^\omega$ coming from the ultra-coproduct of the extremal traces on $B$.   In order to produce elements $a_i$ in Definition \ref{defn:CPOU} so we can use CPoU, it is necessary to work with all traces (not just extremal traces), and thus we must ask for approximate solutions in general (not necessarily factorial) GNS representations.  
Every time we use Lemma \ref{lem:RealizingTypes} here in place of \cite[Lemma 3.18]{BBSTWW} we are able to verify this stronger hypothesis (although it requires working with finite von Neumann algebras in place of just II$_1$ factors).
\end{remark}

\begin{proof}[Proof of Lemma \ref{lem:RealizingTypes}.]
Note that \eqref{eq:RealizingTypesConclusion} is the same as 
\begin{equation}
\|h_m(a_1,\dots,a_{r_m},y_1,\dots,y_{s_m})\|_{2,\tau}=0,\quad \tau \in T_\omega(B),
\end{equation}
(by \eqref{eq:TraceKernelEq}). Just as in the proof of \cite[Lemma 3.18]{BBSTWW}, by Kirchberg's $\epsilon$-test, it suffices to find for each $k \in \mathbb N$ and $\epsilon>0$, contractions $(y_i)_{i=1}^\infty \in B^\omega$ such that
\begin{equation}
\|h_m(a_1,\dots,a_{r_m},y_1,\dots,y_{s_m})\|_{2,\tau} \leq \epsilon,\ m=1,\dots,k,\ \tau \in T_\omega(B).
\end{equation}
So, let us fix $k\in \mathbb N$ and $\epsilon>0$.

By hypothesis, for each $\tau \in \overline{T_\omega(B)}$ there exist contractions $(y^\tau_i)_{i=1}^\infty$ in $\pi_\tau( B^\omega)''$ such that 
\begin{equation}\label{e4.4}
\|h_m(a_1,\dots,a_{r_m},y^\tau_1,\dots,y^\tau_{s_m}))\|^2_{2,\tau} < \tfrac{\epsilon^2}{k},\quad m=1,\dots,k.
\end{equation}
Using Kaplansky's density theorem, we may in fact assume each $y^\tau_i$ is in $B^\omega$ while retaining \eqref{e4.4}.
Set
\begin{equation}
\label{eq:RealizingTypesbtauDef} b^\tau := \ssum_{m=1}^k |h_m(a_1,\dots,a_{r_m},y^\tau_1,\dots,y^\tau_{s_m})|^2 \in B^\omega, \quad\tau\in \overline{T_\omega(B)},\end{equation}
so that
\begin{equation} \tau(b^\tau) = \ssum_{m=1}^k \|h_m(a_1,\dots,a_m,y^\tau_1,\dots,y^\tau_m)\|_{2,\tau}^2 < \epsilon^2,\quad \tau\in \overline{T_\omega(B)}. \end{equation}
By continuity and compactness, there are $\tau_1,\dots,\tau_n \in \overline{T_\omega(B)}$ such that
\begin{equation}
\epsilon^2>\sup_{\tau\in T_\omega(B)}\min_{i=1,\dots,n}\tau(b^{\tau_i}).
\end{equation}
By CPoU, there exist pairwise orthogonal projections 
\begin{equation}
\label{e4.9} p_1,\dots,p_n\in B^\omega\cap \{y^{\tau_1}_i,\dots,y^{\tau_n}_i,a_i:i\in\mathbb N\}' 
\end{equation}
such that
\begin{equation}
\label{eq:RealizingTypesCPoU} \ssum_{j=1}^n p_j = 1_{B^\omega} \quad \text{and}\quad \tau(b^{\tau_j}p_j) \leq \epsilon^2\tau(p_j),\quad \tau \in T_\omega(B). \end{equation}
Define $y_i:= \ssum_{j=1}^n p_jy^{\tau_j}_i$ for $i\in\mathbb N$.
Then, from \eqref{e4.9} and since the $p_j$ are a partition of unity consisting of projections,
\begin{align*}
\label{eq:RealizingTypesPolyCommut}
\nonumber |h_m(a_1,\dots,a_{r_m},y_1,\dots,y_{s_m})|^2 &\,= \ssum_{j=1}^n p_j \big|h_m(a_1,\dots,a_{r_m},y_1^{\tau_j},\dots,y_{s_m}^{\tau_j})\big|^2\\
&\stackrel{\eqref{eq:RealizingTypesbtauDef}}\leq \ssum_{j=1}^np_jb^{\tau_j}.
\end{align*}
Using this, we obtain
\begin{align}
\notag
\|h_m(a_1,\dots,a_{r_m},y_1,\dots,y_{s_m})\|_{2,\tau}^2
&= \tau(|h_m(a_1,\dots,a_{r_m},y_1,\dots,y_{s_m})|^2) \\
\notag&\stackrel{\eqref{eq:RealizingTypesPolyCommut}}\leq \ssum_{j=1}^n \tau(p_jb^{\tau_j}) \\
&\stackrel{\eqref{eq:RealizingTypesCPoU}}\leq \ssum_{j=1}^n \epsilon^2\tau(p_j) = \epsilon^2,
\end{align}
for all $\tau \in T_\omega(B)$, as required.
\end{proof}

Our first use of Lemma \ref{lem:RealizingTypes} is to obtain strict comparison for relative commutants in $B^\omega$.

\begin{lemma}[{cf.\ \cite[Lemma 3.20]{BBSTWW}}]\label{lem:StrictClosureStrictComp}
Let $B$ be a separable, unital $\mathrm{C}^*$-algebra with $T(B)$ non-empty, and with CPoU.
Let $S\subset B^\omega$ be a self-adjoint $\|\cdot\|_{2,T_\omega(B)}$-separable subset, and let $p$ be a projection in the centre of $B^\omega\cap S'$.
Then the $\mathrm{C}^*$-algebra $p(B^\omega\cap S')$ has strict comparison of positive elements by bounded traces, as in \cite[Definition 1.5]{BBSTWW}.
\end{lemma}
\begin{proof}
This is obtained by following the proof \cite[Lemma 3.20]{BBSTWW} replacing $\mathcal M^\omega$ by $B^\omega$ and using Lemma \ref{lem:RealizingTypes} in place of \cite[Lemma 3.18]{BBSTWW}.  There are two things to note. 

Firstly, strict comparison is a property defined at the level of matrix amplifications. The first paragraph of the proof of Lemma 3.20 of \cite{BBSTWW}, notes that without loss of generality we do not need to perform this matrix amplification; this is equally valid in our setting as $M_k(B^\omega)\cong (M_k(B))^\omega$, and $M_k(B)$ inherits CPoU from $B$ by Lemma \ref{CPOU:Matrix}.  

Secondly, to verify the existence of approximate local solutions, \eqref{eq:RealizingTypesFibre} as in the last two paragraphs of the proof of \cite[Lemma 3.20]{BBSTWW}, we only require that $\pi_\tau(B^\omega)''$ is a finite von Neumann algebra,\footnote{In contrast to \cite{BBSTWW}, where the relevant $\pi_\tau(M^\omega)$ is automatically a von Neumann algebra, in our situation we must pass to the von Neumann closure, as $\pi_\tau(B^\omega)$ is not generally a von Neumann algebra.} so the additional traces required in the hypothesis of Lemma \ref{lem:RealizingTypes} (as discussed in Remark \ref{rmk:UsingRealizingTypes} (iii)) cause no difficulties.
\end{proof}

Our next goal is to use CPoU to determine all traces on relative commutants $B^\omega\cap\phi(A)'$, where $A$ is a separable, nuclear $\mathrm{C}^*$-algebra and $\phi$ is a $^*$-homomorphism.  The compact tracial boundary version of this result (phrased in terms of McDuff $\mathrm{W}^*$-bundles) is \cite[Proposition 3.22]{BBSTWW}, obtained by gluing together the fibrewise result of \cite[Lemma 3.21]{BBSTWW}.  In order to clarify the Hahn--Banach argument used at the very last part of these two results,\footnote{The results (\cite[Lemma 3.21 and Proposition 3.22]{BBSTWW}) are correct as stated, but the presentation of the argument using the Hahn--Banach theorem at the end of their proofs was inappropriately terse.} we give full details. We begin by extracting a Hahn--Banach argument from the proof of \cite[Theorem 8]{Oz13}.

\begin{lemma}[{cf.\ \cite[Theorem 8]{Oz13}}]\label{HB}
Let $A$ be a $\mathrm{C}^*$-algebra, and let  $T_0\subset T(A)$ be a collection of traces with the property that for all $z\in A$, 
\begin{equation}
\sup_{\tau\in T_0}|\tau(z)|=\sup_{\tau\in T(A)}|\tau(z)|.
\end{equation}
Then the weak$^*$-closed convex hull of $T_0$ is $T(A)$.
\end{lemma}
\begin{proof}
Suppose that this is not the case. Then there exists $\tau\in T(A)\setminus\overline{\mathrm{co}(T_0)}$. By Hahn--Banach, let $z_0=z_0^* \in A$ be such that $\sup_{\sigma\in T_0}\sigma(z_0)<\tau(z_0)$.
By hypothesis, 
\begin{equation} \alpha:=\sup_{\sigma\in T_0}|\sigma(z_0)|-\tau(z_0)\geq 0. \end{equation}
Choose $a\in A_+$ such that $\tau(a)>\alpha$ and $\|a\|<\alpha+(\tau(z_0)-\sup_{\sigma\in T_0}\sigma(z_0))$. Set $z:=z_0+a$, so that for $\rho\in T_0$,
\begin{align}
\notag
-&\sup_{\sigma\in T_0}|\sigma(z_0)|\leq\rho(z_0)\leq\rho(z)\leq\rho(z_0)+\|a\|\\
&<\rho(z_0)+\alpha+\tau(z_0)-\sup_{\sigma\in T_0}\sigma(z_0)
<\tau(z_0)+\alpha=\sup_{\sigma\in T_0}|\sigma(z_0)|.
\end{align}
Thus,
$\sup_{\sigma\in T_0}|\sigma(z)|\leq\sup_{\sigma\in T_0}|\sigma(z_0)|$,
and so
\begin{align}
\notag
\tau(z)=\tau(z_0)+\tau(a)&>\tau(z_0)+\sup_{\sigma\in T_0}|\sigma(z_0)|-\tau(z_0)\\
&=\sup_{\sigma\in T_0}|\sigma(z_0)|\geq\sup_{\sigma\in T_0}|\sigma(z)|,
\end{align}
contradicting the hypothesis.
\end{proof}

We now turn to the fibrewise statement for traces on relative commutants.  The key difference between this and the original statement of \cite[Lemma 3.21]{BBSTWW} is that we also need to consider elements which are uniformly small in trace, not just those which are zero in all traces.

\begin{lemma}[{cf.\ \cite[Lemma 3.21]{BBSTWW}}]
\label{lem:vnCommTraces}
Let $\mathcal M$ be a finite von Neumann algebra with faithful normal trace $\tau_{\mathcal M}$. Let $\mathcal B\subset \mathcal M$ be an injective von Neumann subalgebra with separable predual. 
Set $\mathcal N:=\mathcal M \cap \mathcal B'\cap (1_{\mathcal M}-1_{\mathcal B})^\perp$.
Define $T_0$ to be the set of all traces on $\mathcal N$ of the form $\tau(b\,\cdot)$ where $\tau \in T(\mathcal M)$ is a normal trace and $b \in \mathcal B_{+}$ has $\tau(b)=1$. Suppose that $\delta>0$ and $z \in \mathcal N$ is a contraction such that $|\rho(z)|\leq\delta$ for all $\rho \in T_0$. Set $K := 12 \cdot 12 \cdot (1+\delta)$.
Then for any finite subset $\mathcal F$ of $\mathcal B$ and $\epsilon > 0$, there exist contractions $w,x_1,\dots,x_{10},y_1,\dots,y_{10} \in \mathcal M\cap (1_{\mathcal M}-1_{\mathcal B})^\perp$, such that
$\|[w,a]\|_{2,\tau_{\mathcal M}}, \|[x_i,a]\|_{2,\tau_{\mathcal M}}, \|[y_i,a]\|_{2,\tau_{\mathcal M}} < \epsilon$ for all $a \in \mathcal F$, and
\begin{equation}
\label{eq:vnCommTraces10Comm}
\big\| z - \delta w - K\ssum_{i=1}^{10} [x_i,y_i]\big\|_{2,\tau_{\mathcal M}} < \epsilon.
\end{equation}

If $(\mathcal M,\tau_\mathcal M)$ is an ultraproduct of tracial von Neumann algebras,\footnote{\label{TVNAFT}This means that we have a sequence $(\mathcal M_n)_{n=1}^\infty$ of finite von Neumann algebras with a distinguished faithful trace $\tau_n \in T(\mathcal M_n)$ for each $n$, and $\mathcal M:=\prod^\omega \mathcal M_n$ is the tracial von Neumann ultrapower, i.e. the quotient of $\prod_{n=1}^\infty\mathcal M_n$ by $\{(x_n)_{n=1}^\infty \in\prod_{n=1}^\infty\mathcal M_n:\lim_{n\to\omega}\tau_n(x_n^*x_n)=0\}$.  Then $\mathcal M$ is a finite von Neumann algebra and carries the distinguished faithful normal trace $\tau_\mathcal M$ arising from the sequence $(\tau_n)_{n=1}^\infty$.} then one can take $w, x_1,\dots,x_{10},y_1,\dots,y_{10}\in\mathcal N$ and have the equality
\begin{equation}
\label{eq:vnCommTraces10Comm2}
 z = \delta w + K\ssum_{i=1}^{10} [x_i,y_i].
\end{equation}  
Moreover, in this case, the weak$^*$-closed convex hull of $T_0$ is $T(\mathcal N)$.
\end{lemma}

\begin{proof}
Throughout the proof, we will write $\|\cdot\|_2$ as shorthand for $\|\cdot\|_{2,\tau_{\mathcal M}}$.
Assume without loss of generality that $\mathcal F$ consists of contractions.
As $\mathcal B$ is hyperfinite (by Connes' theorem \cite{Co76}), we may find a finite dimensional subalgebra $B$ of $\mathcal B$ with $1_B=1_{\mathcal B}$ such that for all $a\in\mathcal{F}$  there exists $b \in B$ such that $\|a - b\|_2 < \tfrac{\epsilon}{2}$. We shall arrange that $w,x_i,y_i \in \mathcal M \cap B'\cap (1_{\mathcal M}-1_B)^\perp$; it will then follow that $\|[w,a]\|_2, \|[x_i,a]\|_2, \|[y_i,a]\|_2 < \epsilon$ for all $a \in \mathcal F$.

Let $p_1,\dots,p_l$ be the minimal central projections in $B$, so that $\mathcal M\cap B'\cap (1_{\mathcal M}-1_B)^\perp$ decomposes as the direct sum $\bigoplus_{k=1}^lp_k(\mathcal M\cap B')$.   Identify the centre of $\mathcal M$ with $L^\infty(Y,\mu)$, where $\mu$ induces $\tau_\mathcal M$ and let $E:\mathcal M\rightarrow L^\infty(Y,\mu)$ denote the centre-valued trace.  For each $k$, set $Y_k:=\{t\in Y:E(p_k)(t)> (\epsilon/l)^2\}$, and let $q_k:=p_k\chi_{Y_k}\in p_k(\mathcal M\cap B')$.  Note that $p_k-q_k$ is a projection satisfying $\|E(p_k-q_k)\|\leq (\epsilon/l)^2$, and hence $\|p_k-q_k\|_2\leq\epsilon/l$.  Set $z':=z\sum_{k=1}^lq_k$ so that 
\begin{equation} \|z'-z\|_2 \leq \ssum_{k=1}^l \|p_k-q_k\|_2 \leq \epsilon. \end{equation}

We claim that $|\sigma(z')|\leq \delta$ for every trace $\sigma \in T(\mathcal M \cap B'\cap (1_{\mathcal M}-1_B)^\perp)$.
By convexity and density, it suffices to prove this for a normal trace $\sigma$ that is concentrated in a single direct summand $p_k(\mathcal M\cap B')$. Since $p_kB$ is a unital matrix algebra in $p_k\mathcal M p_k$, we have
\begin{equation}
p_k\mathcal M p_k\cong p_k(\mathcal M\cap B')\otimes p_kB,
\end{equation}
and so such a $\sigma$ extends uniquely to a (necessarily normal) trace on $p_k\mathcal M p_k$, also denoted $\sigma$.
Since $E(q_k)$ is bounded away from $0$ on $Y_k$, the tracial functional $\sigma(q_k\,\cdot)$ on $q_k\mathcal M q_k$ extends to a bounded normal tracial functional on $\chi_{Y_k}\mathcal M$,\footnote{To see this, fix $m\in\mathbb N$ with $m>(l/\epsilon)^2$.  Working in the von Neumann algebra $\chi_{Y_k}\mathcal M\otimes M_m$, as the centre-valued trace determines the order on projections, we have $\chi_{Y_k}\otimes e_{11}\precsim q_k\otimes 1_m$. As $\sigma(q_k\,\cdot)$ certainly extends to $q_k\mathcal M q_k\otimes M_m$, this Murray--von Neumann subequivalence can be used to define the trace on $\chi_{Y_k}\mathcal M$.}
and then to a bounded normal tracial functional $\sigma'$ on $\mathcal M$. Note $\sigma'(p_kz)=\sigma'(p_k\chi_{Y_k}z)=\sigma(q_kz)$, and $\sigma'(p_k)=\sigma(q_k)$. Accordingly, using the hypothesis on $z$ for the first inequality, we have
\begin{align}
|\sigma(z')|=|\sigma(q_kz)|=|\sigma'(p_kz)|\leq\delta\sigma'(p_k)=\delta\sigma(q_k)\leq\delta,
\end{align}
establishing the claim.

Now, let $h \in Z\big(\mathcal M\cap B'\cap (1_{\mathcal M}-1_B)^\perp\big)$ be the image of $z'$ under the centre-valued trace on $\mathcal M \cap B'\cap(1_{\mathcal M}-1_B)^\perp$. By the claim, $\|h\|\leq \delta$, so $w := \delta^{-1} h$ is a contraction, and $z'-\delta w=z'-h$ vanishes under this centre-valued trace. Therefore \cite[Theorem 3.2]{FH80} provides contractions\footnote{This is why the constant $K$ appears. Note $\|z'-h\| \leq 1 + \delta$.} $x_{1},\dots,x_{10},y_{1},\dots,y_{10} \in \mathcal M \cap B'\cap(1_{\mathcal M}-1_B)^\perp$, such that
\begin{equation} z'-\delta w = K\ssum_{i=1}^{10} [x_{i},y_{i}]. \end{equation} 
Since $\|z-z'\|_2 < \epsilon$, this gives (\ref{eq:vnCommTraces10Comm}).

Now consider the case when $(\mathcal M,\tau_\mathcal M)$ is an ultraproduct of tracial von Neumann algebras as set out in Footnote \ref{TVNAFT} above.
In this case Kirchberg's $\epsilon$-test enables us to take $x_i,y_i\in\mathcal N$ and get the equality (\ref{eq:vnCommTraces10Comm2}), which in turn ensures that for any trace $\rho\in T(\mathcal N)$, we have $|\rho(z)|\leq\delta$.    Thus $\mathcal N$ (in place of $A$) and $T_0$ satisfy the hypothesis of Lemma \ref{HB}, and so the weak$^*$-closed convex hull of $T_0$ is $T(\mathcal N)$.
\end{proof}

We now use CPoU to obtain a version of the previous result for relative commutants in $B^\omega$.

\begin{proposition}[{cf.\ \cite[Proposition 3.22]{BBSTWW}}]\label{prop:CommTraces}
Let $B$ be a separable, unital $\mathrm{C}^*$-algebra with $T(B)$ non-empty, and with CPoU.
Let $A$ be a separable, unital, nuclear $\mathrm{C}^*$-algebra and $\phi:A\rightarrow B^\omega$ a $^*$-homomorphism.  Set $C:= B^\omega\cap\phi(A)'\cap(1_{B^\omega}-\phi(1_A))^\perp$. Define $T_0$ to be the set of all traces on $C$ of the form $\tau(\phi(a)\,\cdot)$ where $\tau \in T(B^\omega)$ and $a \in A_+$ satisfies $\tau(\phi(a))=1$.

Suppose $\delta>0$ and $z \in C$ is a contraction satisfying $|\rho(z)|\leq \delta$ for all $\rho\in T_0$. Set $K := 12 \cdot 12 \cdot (1+\delta)$. Then there exist contractions $w, x_1,\dots,x_{10},y_1,\dots,y_{10} \in C$, such that
\begin{equation}\label{T3.19:3.37}
  z = \delta w + K\ssum_{i=1}^{10} [x_i,y_i].
\end{equation}
In particular, $T(C)$ is the closed convex hull of $T_0$.
\end{proposition}

\begin{proof}
Fix $z$ and $\delta$ as in the proposition, and a dense sequence $(a_j)_{j=1}^\infty$ in $A$. We will use Lemma \ref{lem:RealizingTypes} to find contractions $w, x_1,\dots,x_{10},y_1,\dots,y_{10}$ in $B^\omega$ satisfying the following conditions 
\begin{align}
\notag	 z - \delta w - K\ssum_{i=1}^{10} [x_i, y_i] &= 0, \\
\notag	 s(1_{B^\omega} - \phi(1_A)) &= 0, \quad s \in \{x_1,\dots,x_{10},y_1,\dots,y_{10},w\}, \\
	[s,\phi(a_j)] &= 0, \quad j \in \mathbb{N},\ s \in \{x_1,\dots,x_{10},y_1,\dots,y_{10},w\},\label{prop:CommTraces.neweq}
\end{align}
which we can encode by a countable sequence of non-commutative polynomials $h_m$ for the purpose of verifying the hypotheses of Lemma \ref{lem:RealizingTypes}. 

Let $\tau \in \overline{T_\omega(B)}$, $\epsilon > 0$ and $k \in \mathbb{N}$. Set $\mathcal M_\tau := \pi_\tau(B^\omega)''$, $\phi_\tau := \pi_\tau \circ \phi:A \rightarrow \mathcal M_\tau$ and  $\mathcal{B} := \phi_\tau(A)''$, where $\pi_\tau$ denotes the GNS representation associated to $\tau$.
 As $A$ is nuclear, $\mathcal B$ is injective.

Given a normal trace $\rho\in T(\mathcal M_\tau)$, and $b\in \mathcal B_+$ with $\rho(b)=1$, find a net $(a_i)_i$ in $A_+$ with $\phi_\tau(a_i)\rightarrow b$ weak$^*$. Then $\rho(\phi_\tau(a_i))\rightarrow 1$, so defining $\tilde{a}_i:=a_i/\rho(\phi_\tau(a_i))$, we have $\phi_\tau(\tilde{a}_i)\rightarrow b$ weak$^*$ and $\rho(\phi_\tau(\tilde{a}_i))=1$ for all $i$.  By hypothesis $|\rho(\phi_\tau(\tilde{a}_i)\pi_\tau(z))|\leq \delta$ for all $i$, so that $|\rho(b\pi_\tau(z))|\leq\delta$.  Lemma \ref{lem:vnCommTraces}, then gives contractions $w^{\tau},x_1^{\tau},\dots,x_{10}^{\tau},y_1^{\tau},\dots,y_{10}^{\tau} \in \mathcal M_\tau \cap (1_{{\mathcal M}_\tau}-1_{\mathcal B})^\perp$ such that 
\begin{equation}
\big\| \pi_\tau(z) - \delta w^{\tau} - K\ssum_{l=1}^{10} [x_l^{\tau},y_l^{\tau}]\big\|_{2,\tau} < \epsilon,
\end{equation}
and $\|[w^\tau,\phi_\tau(a_j)]\|_{2,\tau}, \|[x^\tau_l,\phi_\tau(a_j)]\|_{2,\tau}, \|[y^\tau_l\phi_\tau(a_j)]\|_{2,\tau} < \epsilon$ for all $1 \leq j \leq k$ and $1\leq l\leq 10$. 
Hence the hypotheses of Lemma \ref{lem:RealizingTypes} are satisfied, i.e., we can approximately satisfy finitely many of the conditions from \eqref{prop:CommTraces.neweq} in each trace.  Therefore by Lemma \ref{lem:RealizingTypes}, we can produce the contractions  satisfying \eqref{prop:CommTraces.neweq}, thus establishing (\ref{T3.19:3.37}). 

Since $w$ is a contraction, we have $|\tau(z)| \leq \delta$ for all $\tau \in T(C)$. Hence, the closure of $T_0$ is $T(C)$ by Lemma \ref{HB}.
\end{proof}

With the above results in place, we can now use property (SI) as set up in \cite[Section 4.1]{BBSTWW} (without any tracial boundary assumption) and central surjectivity to lift the structural results for relative commutants in $B^\omega$ back to the $\mathrm{C}^*$-level.
 In particular, using the results above in place of the corresponding results from \cite{BBSTWW}, the proofs in \cite{BBSTWW} give the following omnibus lemma (which was shown in \cite{BBSTWW} in the compact boundary case).  As in \cite{BBSTWW}, from this point on we need the hypotheses that $B$ is simple, and all quasitraces on $B$ are traces (written $QT(B)=T(B)$, and famously automatic when $B$ is exact by Haagerup's work \cite{Ha14}).

\begin{lemma}[{cf.\ \cite[Theorem 4.1]{BBSTWW}}]\label{lem:Omnibus}
Let $B$ be a separable, simple, unital, $\mathcal Z$-stable $\mathrm{C}^*$-algebra with $QT(B)=T(B)\neq \emptyset$ and with CPoU.
Let $A$ be a separable, unital, nuclear $\mathrm{C}^*$-algebra and $\pi:A\rightarrow B_\omega$ a c.p.c.\ order zero map such that $\pi(a)$ is full for each nonzero $a\in A$, and which induces a $^*$-homomorphism $\bar{\pi}:A\rightarrow B_\omega/J_{B}=B^\omega$.  Set
\begin{equation}\label{eq:DefC}
C:=B_\omega\cap \pi(A)'\cap (1_{B_\omega}-\pi(1_A))^\perp,\quad \bar{C}:=C/(C\cap J_{B}).
\end{equation}
Then
\begin{enumerate}
\item[\rm{(i)}] all traces on $C$ factor through $\bar{C}$,
\item[\rm{(ii)}] $C$ has strict comparison of positive elements by traces, and
\item[\rm{(iii)}] the traces on $C$ are the closed convex hull of traces of the form $\tau(\pi(a)\,\cdot)$ for $\tau\in T(B_\omega)$ and $a\in A_+$ with $\tau(\pi(a))=1$.
\end{enumerate}
\end{lemma}

\begin{proof}
(i) is just a repeat of \cite[Theorem 4.1(i)]{BBSTWW}, and CPoU is not required.

For (ii), we use the proof from \cite{BBSTWW} with the following small modifications.
In place of the paragraph following Eq.\ (4.57) in \cite{BBSTWW}, use Lemma \ref{lem:StrictClosureStrictComp} 
to obtain strict comparison for $\bar{C}$;
the embedding of $M_k$ into $B^\omega\cap \bar{\pi}(A)'\cap\{\bar{c}\}'$ in Eq.\ (4.58) of \cite{BBSTWW} is obtained from Proposition \ref{prop:GammaZ} in place of \cite[Remark 3.13]{BBSTWW});
and Lemma \ref{lem:CentralSurjection} is used in place of \cite[Lemma 3.10]{BBSTWW} to justify Eq.\ (4.59) of \cite{BBSTWW}.

Finally, (iii) follows from (i), Lemma \ref{lem:CentralSurjection}, and Proposition \ref{prop:CommTraces}.
\end{proof}

We now have all the tools in place to obtain the main classification lemma for order zero maps required to obtain our nuclear dimension estimates.  The following removes the tracial boundary hypothesis from \cite[Theorem 5.5]{BBSTWW} (though for simplicity we stick to a single algebra $B$ rather than the sequence of algebras $(B_n)_{n=1}^\infty$ given in \cite{BBSTWW}).  

\begin{lemma}[{cf.\ \cite[Theorem 5.5]{BBSTWW}}]
\label{lem:ClassMaps}
Let $B$ be a separable, simple, unital, $\mathcal Z$-stable $\mathrm{C}^*$-algebra with $QT(B)=T(B)\neq \emptyset$ and with CPoU.
Let $A$ be a separable, unital, nuclear $\mathrm{C}^*$-algebra, let $\phi_1:A\rightarrow B_\omega$ be a $^*$-homomorphism such that $\phi_1(a)$ is full in $B_\omega$ for each nonzero $a\in A$, and let $\phi_2:A\rightarrow B_\omega$ a c.p.c.\ order zero map such that
\begin{equation}
\label{eq:ClassMapsTrace}
\tau\circ\phi_1=\tau\circ\phi_2^m,\quad \tau\in T(B_\omega),\ m\in\mathbb N.\footnote{Here $\phi_2^m$ denotes the order zero functional calculus, cf. \cite[Corollary 4.2]{WZ09}.}
\end{equation}
Let $k\in \mathcal Z_+^1$ have spectrum $[0,1]$, and define c.p.c.\ order zero maps $\psi_i:A\rightarrow (B\otimes\mathcal Z)_\omega$ by $\psi_i(\cdot):=\phi_i(\cdot)\otimes k$.  Then $\psi_1$ and $\psi_2$ are unitarily equivalent in $(B\otimes\mathcal Z)_\omega$.
\end{lemma}

\begin{proof}
The proof of \cite[Theorem 5.5]{BBSTWW} works verbatim to give Lemma \ref{lem:ClassMaps}, using Lemma \ref{lem:Omnibus} in place of \cite[Theorem 4.1]{BBSTWW}.  Note that the key technical result (\cite[Theorem 5.1]{BBSTWW}) used in the argument is set up with the hypothesis that the algebra $C$ of \cite[Eq.\ (5.1)]{BBSTWW} has strict comparison of positive elements by traces. When we apply this, $C$ is a relative commutant of $M_2(B_\omega)\cong (M_2(B))_\omega$ but otherwise of the form (\ref{eq:DefC}).  Since $M_2(B)$ inherits CPoU from $B$ by Lemma \ref{CPOU:Matrix}, the hypotheses of Lemma \ref{lem:Omnibus} hold. So \cite[Theorem 5.1]{BBSTWW} can be applied just as in \cite{BBSTWW}.
\end{proof}

We end this section by noting that we can remove the tracial boundary hypothesis from the $2$-coloured classification result \cite[Corollary 6.5]{BBSTWW}.\footnote{The tracial boundary hypothesis can also be removed from \cite[Theorem 6.2]{BBSTWW}; for conciseness we do not state this here.}  Recall from \cite[Definition 6.1]{BBSTWW} that unital $^*$-homomorphisms $\phi,\psi:A\rightarrow B$ are said to be approximately $n$-coloured equivalent if there exists $w^{(0)},\dots,w^{(n-1)}\in B_\omega$ such that $w^{(i)}{}^*w^{(i)}\in B_\omega\cap \psi(A)'$ and $w^{(i)}w^{(i)}{}^*\in B_\omega\cap \phi(A)'$ for each $i=0,\dots,n-1$ and
\begin{equation}
\phi(a)=\sum_{i=0}^{n-1}w^{(i})\psi(a)w^{(i)}{}^*,\ \psi(a)=\sum_{i=0}^{n-1}w^{(i)}{}^*\phi(a)w^{(i)},\quad a\in A.
\end{equation}
Recall too, that under the hypotheses of the theorem which follows, $B$ has CPoU when it is nuclear by Theorem \ref{thm:mainCPoU}.  The proofs from \cite[Section 6]{BBSTWW} work verbatim for Theorem \ref{colouredthm}, using Lemma \ref{lem:ClassMaps} in place of \cite[Theorem 5.5]{BBSTWW}.

\begin{theorem}[{cf.\ \cite[Corollary 6.5]{BBSTWW}}]\label{colouredthm}
Let $A$ be a separable, unital, nuclear $\mathrm{C}^*$-algebra, and let $B$ be a separable, simple, unital, $\mathcal Z$-stable $\mathrm{C}^*$-algebra such that $QT(B)=T(B)\neq \emptyset$, and with CPoU.
Let $\phi_1,\phi_2:A\rightarrow B$ be unital $^*$-homomorphisms such that $\phi_1$ is injective. Then the following are equivalent:
\begin{enumerate}
\item[\rm{(i)}] $\tau\circ\phi_1=\tau\circ\phi_2$ for all $\tau\in T(B)$;
\item[\rm{(ii)}] $\phi_1$ and $\phi_2$ are approximately $n$-coloured equivalent for some $n\in\mathbb N$ with $n\geq 2$;
\item[\rm{(iii)}] $\phi_1$ and $\phi_2$ are approximately $2$-coloured equivalent.
\end{enumerate}
\end{theorem}

\section{Nuclear dimension}\label{S5}

\noindent
The final ingredient we need to prove Theorem \ref{thm:MainThm1} and its consequences, is the existence result (i) discussed in the outline `From $\mathcal Z$-stability to finite nuclear dimension' 
in the introduction.  The corresponding result with a compact tracial boundary assumption is \cite[Lemma 7.4]{BBSTWW}. In order to use CPoU to extend this result to general trace simplices, we first handle the case of a single trace using a strategy from \cite{BCW16}.

\begin{lemma}
	\label{lem:OneTraceMap}
	Let $A$ be a separable, unital, nuclear $\mathrm{C}^*$-algebra.
	Let $\mathcal F \subset A$ be a finite set, let $\epsilon>0$, and let $\tau \in T(A)$.
	Then there exist a finite dimensional $\mathrm{C}^*$-algebra $F$, a c.p.c.\ map $\theta:A \to F$, and a c.p.c.\ order zero map $\eta:F \to A$ such that
	\begin{align}
	\label{eq:OneTraceMap1}
	\|\theta(x)\theta(y)\| &< \epsilon \quad && \text{for }x,y \in \mathcal{F} \text{ satisfying }xy=0\text{, and} \\ 
	\label{eq:OneTraceMap2}
	\|\eta \circ \theta(x)-x\|_{2,\tau} &< \epsilon &&\text{for }x \in \mathcal{F}.
	\end{align}
	If all traces on $A$ are quasidiagonal, then in place of \eqref{eq:OneTraceMap1} we may arrange that
	\begin{equation}
	\label{eq:OneTraceMap3} \|\theta(x)\theta(y)-\theta(xy)\| < \epsilon \quad \text{for }x,y \in \mathcal{F}. \end{equation}
\end{lemma}

\begin{proof}
Set $C:=(C_0((0,1])\otimes A)^\sim$, and define a c.p.c.\ order zero map $\psi:A \to C$ by $\psi(a):=\id_{(0,1]}\otimes a$. By \cite[Proposition 3.2]{BCW16}, every trace on $C$ is quasidiagonal.\footnote{That proposition shows the result for $C_0((0,1])\otimes A$, but it directly follows for the unitisation as well, by unitising.} Therefore \cite[Lemma 2.5]{BCW16} applies to $C$.  

Since the GNS representation $\pi_{\mathrm{ev}_1 \otimes \tau}(C)''$, where $\mathrm{ev}_1$ denotes evaluation at $1$ on $C_0((0,1])^\sim$, is a direct summand of the finite part of $C^{**}$, it follows by \cite[Lemma 2.5]{BCW16} that there are a finite dimensional $\mathrm{C}^*$-algebra $F$, a c.p.c.\ map $\tilde\theta:C \to F$, and a $^*$-homomorphism $\tilde\eta:F \to \pi_{\delta_1\otimes\tau}(C)''$ such that
	\begin{align}
	\|\tilde\theta(x)\tilde\theta(y)-\tilde\theta(xy)\| &< \epsilon \quad && \text{for }x,y \in \psi(\mathcal F), \text{ and} \label{eq:OneTraceMap.new1}\\
	\|\tilde\eta \circ \tilde\theta(x)-\pi_{\mathrm{ev}_1 \otimes \tau}(x)\|_{2,\mathrm{ev}_1 \otimes \tau} &< \epsilon &&\text{for }x \in \psi(\mathcal{F}). \label{eq:OneTraceMap.new2}
	\end{align}
Define $\theta:=\tilde\theta\circ \psi:A \to F$, so that \eqref{eq:OneTraceMap1} is a consequence of \eqref{eq:OneTraceMap.new1}.

Since $\mathrm{ev}_1\otimes\id_A:C \to A$ is surjective, a routine computation with the GNS-construction gives an isomorphism $\pi_{\mathrm{ev}_1\otimes\tau}(C)''\cong\pi_\tau(A)''$ making
	\begin{equation} \xymatrix{C \ar[r]^{\mathrm{ev}_1\otimes\id_A} \ar[d]^{\pi_{\mathrm{ev}_1\otimes\tau}} & A \ar[d]^{\pi_\tau} \\ \pi_{\mathrm{ev}_1\otimes\tau}(C)'' \ar[r]^{\cong} & \pi_\tau(A)''} \end{equation}
commute.
Thus we may view the codomain of $\tilde\eta$ as $\pi_\tau(A)''$.
Moreover, for $x \in \mathcal F$, as $(\mathrm{ev}_1\circ\id_A)\circ\psi(x)=x$, \eqref{eq:OneTraceMap.new2} gives
\begin{equation}
\|\tilde\eta\circ\theta(x)-\pi_\tau(x)\|_{2,\tau} = \big\|\tilde\eta\circ\tilde\theta(\psi(x))-\pi_{\mathrm{ev}_1 \otimes \tau}(\psi(x))\big\|_{2,\mathrm{ev}_1\otimes \tau} < \epsilon. \end{equation}
By \cite[Lemma 1.1]{HKW12}, we may approximate the map $\tilde\eta$ in the point-strong$^*$ topology (equivalently point-$\|\cdot\|_{2,\tau}$) by a c.p.c.\ order zero map $\eta$ into $A$, so that $\|\eta\circ\theta(x)-x\|_{2,\tau} < \epsilon$ for $x \in \mathcal F$.

In the case that all traces on $A$ are quasidiagonal, we simply apply \cite[Lemma 2.5]{BCW16} directly to $A$ (instead of $C$) to get $\tilde\eta$ and $\theta$, and then again use \cite[Lemma 1.1]{HKW12} to get $\eta$.
\end{proof}

We now glue the previous lemma over $T(A)$ using CPoU to obtain our existence result. Note that the conclusion of Lemma \ref{lem:NiceFactoring} is stronger than its counterpart in \cite{BBSTWW}, since it exactly factorises the map $A \to A^\omega$ (as opposed to producing a factorisation which agrees on traces).

\begin{lemma}[{cf.\ \cite[Lemma 7.4]{BBSTWW}}]
\label{lem:NiceFactoring}
	Let $A$ be a separable, unital, nuclear $\mathrm{C}^*$-algebra with $T(A)\neq\emptyset$ and CPoU.
	Then there exists a sequence of c.p.c.\ maps $\phi_n:A \to A$ which factor through finite dimensional algebras $F_n$ as
	\begin{equation}
	\label{eq:NiceFactoring1}
	\xymatrix{A\ar[dr]_{\theta_n}\ar[rr]^{\phi_n}&&A\\&F_n \ar[ur]_{\eta_n}}
	\end{equation}
	with $\theta_n$ c.p.c.\ and $\eta_n$ c.p.c.\ order zero in such a way that the induced map $(\theta_n)_{n=1}^\infty:A \to \prod_\omega F_n$ is order zero and the induced map $\Phi=(\phi_n)_{n=1}^\infty:A\rightarrow A^\omega$ agrees with the diagonal map $A \to A^\omega$.
	
	If all traces on $A$ are quasidiagonal, then we may arrange that $(\theta_n)_{n=1}^\infty$ is unital.
\end{lemma}

\begin{proof}
By Kirchberg's $\epsilon$-test (similar to the application of the $\epsilon$-test in the proof of \cite[Lemma 7.4]{BBSTWW}), it suffices to show that for a finite set $\mathcal F \subset A$ and a tolerance $\epsilon>0$, there is a sequence of c.p.c.\ maps $\phi_n:A \to A$ which factor through finite dimensional algebras $F_n$ as in \eqref{eq:NiceFactoring1} with $\theta_n$ c.p.c.\ and $\eta_n$ c.p.c.\ order zero such that
\begin{align}
\label{eq:NiceFactoring2}
	\|\theta_n(x)\theta_n(y)\| &< \epsilon \quad && \text{for }x,y \in \mathcal{F} \text{ satisfying }xy=0\text{, and} \\
\label{eq:NiceFactoring3}
	\|\Phi(x)-x\|_{2,T_\omega(A)} &\leq \epsilon \quad && \text{for }x \in \mathcal F,
\end{align}
where $\Phi:A\rightarrow A^\omega$ is the map induced by $(\eta_n\circ\theta_n)_{n=1}^\infty$. In fact, we will arrange for all the $F_n$ to be the same finite dimensional algebra $F$, and all the $\theta_n$ to be the same map $\theta$.

Fix a finite set $\mathcal F \subset A$ and $\epsilon>0$.
For each $\tau \in T(A)$, by Lemma \ref{lem:OneTraceMap}, there are a finite dimensional algebra $F_\tau$, a c.p.c.\ map $\theta_\tau:A \to F_\tau$ satisfying \eqref{eq:NiceFactoring2}, and a c.p.c.\ order zero map $\eta_\tau:F_\tau \to A$ such that 
\begin{equation}
\|\eta_\tau\circ\theta_\tau(x)-x\|_{2,\tau}^2<\epsilon^2/|\mathcal F|,\quad x\in \mathcal F.
\end{equation}
Define 
\begin{equation}\label{lem:NiceFactoring.neweq1}
a_\tau := \ssum_{x\in \mathcal F}|\eta_\tau(\theta_\tau(x))-x|^2 \in A_+,
\end{equation}
 so that $\tau(a_\tau) < \epsilon^2$.

Continuity and compactness give $\tau_1,\dots,\tau_k \in T(A)$ such that 
\begin{equation}\label{lem:NiceFactoring.neweq2}
 \min \{\tau(a_{\tau_1}),\dots,\tau(a_{\tau_k})\} < \epsilon^2,\quad \tau \in T(A). \end{equation}
Viewing $a_{\tau_1},\dots,a_{\tau_k}$ as elements of $A^\omega$, we apply CPoU\footnote{Note that $\min\{\tau(a_{\tau_1}),\dots,\tau(a_{\tau_k})\}<\epsilon^2$ for all $\tau\in T_\omega(A)$.} to get a partition of unity consisting of projections $e_1,\dots,e_k \in A^\omega \cap A'$ such that
\begin{equation} \tau(e_ia_{\tau_i}) \leq \epsilon^2\tau(e_i), \quad \tau \in T_\omega(A),\ i=1,\dots,k. \label{lem:NiceFactoring.neweq3} \end{equation}
	Define $F:=\bigoplus_{i=1}^k F_{\tau_i}$, $\theta:=\bigoplus_{i=1}^k \theta_{\tau_i}: A \to F$, and $\eta: F \to A^\omega$ by
\begin{equation} \eta(x_1,\dots,x_k) := \ssum_{i=1}^k e_i\eta_{\tau_i}(x_i),\quad x_i\in F_{\tau_i}. \end{equation}
Since the $e_i$ are orthogonal projections commuting with the images of the c.p.c\ order zero maps $\eta_{\tau_i}$, it follows that $\eta$ is c.p.c.\ order zero.
By projectivity of c.p.c.\ order zero maps with finite dimensional domains,\footnote{Essentially due to Loring in \cite[Theorem 4.9]{Lo93}, but see \cite[Proposition 1.2.4]{Wi09} for the form we use.} $\eta$ may be lifted to a sequence of c.p.c.\ order zero maps $\eta_n:F \to A$.

It is evident that $\theta_n=\theta$ satisfies \eqref{eq:NiceFactoring2}, since each map $\theta_{\tau_i}$ has the same property.
Next, for $x\in\mathcal F$ and $\tau \in T_\omega(A)$, we have
\begin{align}
\notag
\|\Phi(x)-x\|_{2,\tau}^2
\notag&\stackrel{\phantom{\eqref{lem:NiceFactoring.neweq1}}}= |\eta(\theta(x))-x\|_{2,\tau}^2\\
\notag&\stackrel{\phantom{\eqref{lem:NiceFactoring.neweq1}}}= \ssum_{i=1}^k \tau(e_i|\eta_{\tau_i}(\theta_{\tau_i}(x))-x|^2) \\
&\stackrel{\eqref{lem:NiceFactoring.neweq1}}\leq \ssum_{i=1}^k \tau(e_ia_{\tau_i}) 
\stackrel{\eqref{lem:NiceFactoring.neweq3}}\leq \ssum_{i=1}^k \tau(e_i)\epsilon^2 = \epsilon^2.
\end{align}
Therefore, $\|\Phi(x)-x\|_{2,T_\omega(A)} \leq \epsilon$, for all $x\in\mathcal F$, as required.

If all traces are quasidiagonal, then the above argument (using the last line of Lemma \ref{lem:OneTraceMap}) yields the stronger conclusion that the induced map $(\theta_n)_{n=1}^\infty:A \to \prod_\omega F_n$ is a $^*$-homomorphism.
Hence by cutting down by the image of $1_A$ under this map (which can be lifted to a projection in $\prod_{n=1}^\infty F_n$), we can arrange it to be unital.
\end{proof}

We now have all the pieces in place to obtain Theorem \ref{PropG}, and deduce its consequences. This is a matter of using the existence and uniqueness results established above using CPoU in place of the versions of these results with a compact tracial boundary assumption in \cite{BBSTWW}.

\begin{proof}[Proof of Theorem \ref{PropG}] 
The proof is almost exactly as in \cite[Theorem 7.5]{BBSTWW} (the corresponding result with compact boundary), using Lemma \ref{lem:ClassMaps} and Lemma \ref{lem:NiceFactoring} in place of \cite[Theorem 5.5]{BBSTWW} and \cite[Lemma 7.4]{BBSTWW}.\footnote{Recall that the maps produced in Lemma \ref{lem:NiceFactoring} satisfy a stronger conclusion than those of \cite[Lemma 7.4]{BBSTWW}.}

For the reader's convenience, we spell out a proof of the nuclear dimension component of the theorem explicitly.
Let $(\phi_n:A\rightarrow A)_{n=1}^\infty$ be the sequence of maps produced by Lemma \ref{lem:NiceFactoring}.  By construction, each $\phi_n$ has nuclear dimension zero,\footnote{Strictly speaking the definition of nuclear dimension in \cite[Definition 2.2]{TikW} is only made for $^*$-homomorphisms; for the purpose of this proof, we use exactly the same definition for c.p.c.\ order zero maps.} and the induced map $\tilde{\Phi}:A\rightarrow A_\omega$ is c.p.c.\ order zero. Following $\tilde{\Phi}$ with the quotient map  $q:A_\omega\rightarrow A^\omega$, gives the inclusion $A\hookrightarrow A^\omega$.

Let $h\in \mathcal Z_+^1$ have spectrum $[0,1]$.  We apply Lemma \ref{lem:ClassMaps} with $B=A$ to the diagonal embedding $A\rightarrow A_\omega$ and $\tilde{\Phi}:A\rightarrow A_\omega$ in place of $\phi_1$ and $\phi_2$, once with $k:=h$ and once with $k:=1_{\mathcal Z}-h$.
Simplicity of $A$ ensures that the diagonal embedding satisfies the fullness requirement, while the condition \eqref{eq:ClassMapsTrace} follows from the form of $q\circ\tilde{\Psi}$.  Accordingly there are unitaries $u^{(0)},u^{(1)}\in (A\otimes\mathcal Z)_\omega$, which lift to sequences of unitaries $(u^{(0}_n)_{n=1}^\infty$ and $(u^{(1)}_n)_{n=1}^\infty$ in $A\otimes\mathcal Z$ such that for all $a\in A$, we have
\begin{align}\notag
a\otimes 1_{\mathcal Z}&=a\otimes h + a\otimes (1_{\mathcal Z}-h)\\
&=\lim_{n\rightarrow\omega} (u_n^{(0)}(\phi_n(a)\otimes h)u_n^{(0)}{}^*+u_n^{(1)}(\phi_n(a)\otimes (1_{\mathcal Z}-h))u_n^{(1)}{}^*).
\end{align}
Noting that each of the maps $u_n^{(0)}(\phi_n(\cdot)\otimes h)u_n^{(0)}{}^*$ and $u_n^{(1)}(\phi_n(\cdot)\otimes (1_\mathcal Z-h))u_n^{(1)}{}^*$ has nuclear dimension zero, it follows that $\id_A\otimes1_{\mathcal Z}:A\rightarrow A\otimes \mathcal Z$ has nuclear dimension at most $1$.  As $A$ is $\mathcal Z$-stable, and $\mathcal Z$ is strongly self-absorbing, $\dimnuc(A)\leq 1$ by \cite[Proposition 2.6]{TikW}.
\end{proof}

In the presence of traces, Theorem \ref{thm:MainThm1} is a consequence of Theorems \ref{PropG} and \ref{thm:mainCPoU} (established at the end of Section \ref{S3}). In the absence of traces, Theorem \ref{thm:MainThm1} is \cite[Corollary 9.9]{BBSTWW} (as in this case $A$ is a Kirchberg algebra by \cite[Corollary 5.1]{Ro04}).

Theorem \ref{ThmA} is then a combination of our Theorem \ref{thm:MainThm1} for (ii)$\Rightarrow$(i) and the main result of \cite{Wi12} for (i)$\Rightarrow$(ii). 

\begin{proof}[Proof of Corollary \ref{CorC}]
Firstly recall that a $\mathrm{C}^*$-algebra is AF\footnote{For non-separable $\mathrm{C}^*$-algebras, there are various potential meanings of AF, which are not all equivalent. This corollary uses the local approximation formulation: $A$ is AF if and only if finite subsets of $A$ can be approximated inside finite dimensional subalgebras of $A$.} if and only if it has nuclear dimension zero if and only if it has decomposition rank zero (\cite[Remark 2.2(iii)]{WZ10} and \cite[Example 4.1]{KW04}), so this part of the statement is well known. For the rest of the proof we exclude finite dimensional $\mathrm{C}^*$-algebras, so that Winter's $\mathcal Z$-stability theorem applies.

For separable $A$, these statements are consequences of Theorem \ref{thm:MainThm1} and Winter's $\mathcal Z$-stability theorem (encompassed in Theorem \ref{ThmA}): if the nuclear dimension of $A$ is finite, then it is $\mathcal Z$-stable, so $A$ has nuclear dimension at most $1$. Likewise for decomposition rank, as in this case $A$ is finite with all traces quasidiagonal by \cite[Proposition 8.5]{BBSTWW}.

For the general case, if $A$ has finite nuclear dimension, then for any finite subset $\mathcal{F}$ of $A$ there exists a separable, simple, unital $\mathrm{C}^*$-subalgebra $A_0$ of $A$ containing $\mathcal{F}$ with finite nuclear dimension.\footnote{This is the statement that being simple and of nuclear dimension at most $n$ is separably inheritable (in the sense of \cite[Definition II.8.5.1]{Bla06}). This is a consequence of Proposition 2.6 of \cite{WZ10} (having nuclear dimension at most $n$ is separably inheritable, though this is not the language used in \cite{WZ10}),  Theorem II.8.5.6 of \cite{Bla06} (separable inheritability of simplicity), and Proposition II.8.5.2 (intersection of countably many separably inheritable properties is separably inheritable).} Thus $A_0$ has nuclear dimension at most $1$ by Theorem \ref{thm:MainThm1}. Since $\mathcal F$ was arbitrary, it follows that $A$ has nuclear dimension at most $1$. The same argument works for decomposition rank.
\end{proof}

Corollary \ref{CorD} is a direct replacement of finite nuclear dimension by $\mathcal Z$-stability (using Theorem \ref{thm:MainThm1}) as the regularity hypothesis in the classification theorem (see \cite{Go15}, \cite{EGLN15} and \cite[Corollary D]{TWW17}).  

In Corollaries \ref{CorE} and \ref{NewCor}, the crossed products $C(X)\rtimes G$ are simple when the actions are free and minimal (\cite[Corollary 5.16]{EH67}), and when $G$ is amenable they have the UCT by work of Tu (\cite{Tu}). Thus finite nuclear dimension is the remaining condition which must be checked to obtain classifiability, and so Corollary \ref{NewCor} is a combination of Theorem \ref{thm:MainThm1} and \cite[Theorem 5.4]{CJKMST-D}.

\begin{proof}[Proof of Corollary \ref{CorE}]
In \cite[Theorem 8.1]{KS18}, Kerr and Szab\'o show that every free action of a group $G$ as in Corollary \ref{CorE} on a finite dimensional space is almost finite (building on the zero dimensional case of this result in \cite{DZ17}), and hence if the action is also minimal, the crossed product is $\mathcal Z$-stable by \cite[Theorem 12.4]{K17}.  Finite nuclear dimension for these crossed products is a consequence of Theorem \ref{thm:MainThm1}, and classifiability follows from Corollary \ref{CorD}.\end{proof}

\end{document}